\documentclass{article}
\usepackage{graphicx} 

\usepackage[utf8]{inputenc} 
\usepackage[T1]{fontenc}    
\usepackage{hyperref}       
\usepackage{url}            
\usepackage{booktabs}       
\usepackage{amsfonts}       
\usepackage{nicefrac}       
\usepackage{microtype}      
\usepackage{xcolor}         
\usepackage{graphicx}      
\usepackage{wrapfig}       
\usepackage[margin=1in]{geometry}
\usepackage{amsmath}
\DeclareMathOperator*{\esssup}{ess\,sup}

\newcommand{\RB}{\mathbb{R}}

\newcommand{\NB}{\mathbb{N}}

\newcommand{\GC}{\mathcal{G}}
\newcommand{\FC}{\mathcal{F}}

\newcommand{\EC}{\mathcal{E}}
\newcommand{\AC}{\mathcal{A}}
\newcommand{\VC}{\mathcal{V}}
\newcommand{\BC}{\mathcal{B}}
\newcommand{\SC}{\mathcal{S}}

\newcommand{\CC}{\mathcal{C}}
\newcommand{\WC}{\mathcal{W}}
\newcommand{\MC}{\mathcal{M}}
\newcommand{\ZC}{\mathcal{Z}}
\newcommand{\DC}{\mathcal{D}}
\newcommand{\LC}{\mathcal{L}}

\newcommand{\dom}{\mathrm{dom}}

\newcommand{\Af}{\mathfrak{A}}
\newcommand{\Bf}{\mathfrak{B}}
\newcommand{\Cf}{\mathfrak{C}}
\newcommand{\Gf}{\mathfrak{G}}

\newcommand{\widetildeparallel}[1]{%
  \widetilde{#1}%
  \kern0.1em 
  ^{\vphantom{#1}\parallel}%
}
\newcommand{\widetildeperp}[1]{%
  \widetilde{#1}%
  \kern0.1em 
  ^{\vphantom{#1}\perp}%
}

\newcommand{\tsig}{\widetilde{\sigma}}

\newcommand{\Lip}{\mathrm{Lip}}

\newcommand{\T}{\mathrm{T}}
\newcommand{\D}{\mathrm{D}}
\newcommand{\V}{\mathrm{V}}

\newcommand{\Lb}{\mathbf{L}}

\newcommand{\Hom}{\mathrm{Hom}}

\usepackage{amsmath}
\usepackage{amssymb}
\usepackage{mathtools}
\usepackage{amsthm}
\usepackage{tabularx}
\usepackage{enumitem}
\usepackage{cleveref}

\theoremstyle{plain}
\newtheorem{theorem}{Theorem}[section]
\newtheorem{proposition}[theorem]{Proposition}

\newtheorem{lemma}[theorem]{Lemma}
\newtheorem{corollary}[theorem]{Corollary}

\theoremstyle{definition}
\newtheorem{definition}[theorem]{Definition}

\theoremstyle{remark}

\title{Locally invariant manifolds and foliations}
\author{Lachlan E. MacDonald}
\date{\today}

\begin{document}

\maketitle


\abstract{Invariant manifolds and foliations associated to invariant pseudo-hyperbolic reference geometries have been well-studied in the context of dynamics implemented by a diffeomorphism. In this paper, we provide sufficient conditions for the existence of maximally smooth, \emph{locally} invariant manifolds and foliations containing a given \emph{locally} invariant, pseudo-hyperbolic reference geometry, in the context of dynamics implemented by a smooth map which need not be invertible. The theory is motivated by applications to the developing theory of large-step optimisation.}

\section{Introduction}

This paper proves invariant manifold and foliation results in a pseudo-hyperbolic setting that can be conveniently applied to the theory of ``large-step" optimisation, outside the domain of classical convergence theorems \cite{goldstein, polyak} yet commonly practiced in deep learning \cite{coheneos}. Although the pseudo-hyperbolic theory of invariant manifolds and foliations is by now well-developed, several peculiarities of the large-step optimisation of deep learning models unfortunately place this application outside the explicitly-stated scope of existing literature. To be specific, the reference geometries present in such problems are typically best modelled as only \emph{locally invariant}, rather than invariant/inflowing/overflowing as in classical literature \cite{hpsinvariant,Fenichel1974AsymptoticStability}; moreover, the large step size used in such settings can make the algorithm non-invertible, even locally, unlike the diffeomorphism assumption made in much classical literature \cite{hpsinvariant,Fenichel1974AsymptoticStability}. The purpose of this work is therefore primarily to provide quotable results that can be used in future investigation of large-step algorithms. Secondarily, the author hopes that the general hypotheses invoked in this work and the unified proof technique used herein may be of independent utility for both future research and pedagogical purposes.

\vspace{2mm}

\noindent\textbf{Invariant manifolds and foliations for large-step gradient descent: } To better motivate the paper, consider the problem of finding a minimum of $C^{k+1}$ function $\ell:\RB^p\rightarrow\RB$, with Euclidean gradient and Hessian denoted $\nabla\ell$ and $\nabla^2\ell$ respectively.  A natural way of solving this problem, introduced well over a century ago by Cauchy \cite{Cauchy1847Gradient}, is to iterate the \emph{gradient descent} map
\[
w\mapsto w-\eta\nabla\ell(w),\qquad w\in\RB^p
\]
for some \emph{step size} $\eta>0$. Although it is intuitively clear that the algorithm ought to reduce $\ell$ iteratewise, provided the step size is sufficiently small, convergence guarantees for this method did not arrive until over a century later \cite{goldstein, polyak}, in the form now taught to undergraduates.  Specifically, convergence of the algorithm to a critical point of $\ell$ can be guaranteed if $k\geq 1$, $\ell$ is bounded below, and the step size satisfies the stability condition $\eta<2(\sup_w\lambda_{\max}(\nabla^2\ell(w)))^{-1}$, with $\lambda_{\max}(\nabla^2\ell)$ denoting the largest Hessian eigenvalue. From the dynamical systems perspective, this stability condition is easily seen to be a uniform stability bound on the linearisation of the map, and is a necessary condition for convergence when $\ell$ is quadratic; from an optimisation perspective stability conditions of this kind are one of the essential guarantors of convergence for much of classical optimisation theory \cite{nesterov}; when it holds, $\ell$ necessarily decreases along iterates.

The optimisation of a deep learning model reduces to a minimisation problem of this form, in which $\ell$ denotes the average deviation of the outputs of the model from desired outputs over a training set. Minimising $\ell$ thus corresponds to fitting the model to the data in the training set. In this application, gradient descent (combined with an ever-increasing host of augmentations including stochasticity, preconditioning and momentum) is the canonical choice of optimiser. The theorist may be tempted to draw comfort from the fact that such a famous and well-studied optimiser is used in this setting since it would seem to open the black box of deep learning up to the tools of classical optimisation theory. Such comfort is unfortunately poorly-founded. Recent empirical work \cite{coheneos} indicates that the optimisation of deep learning models is typically practiced in a \emph{large-step} regime in which the classical stability condition is routinely violated, with the largest eigenvalue $\lambda_{\max}(\nabla^2\ell(w))$ initially rising almost monotonically before self-stabilising at around the stability threshold $2/\eta$, with the objective $\ell$ decreasing non-monotonically to convergence. The dynamics underlying this surprising behaviour have received significant attention \cite{damianeos, wu2024eos, macdonaldeos}, but remain poorly understood.

The theorist of deep learning must therefore grapple with the following question: what are the dynamics of \emph{large-step} gradient descent on deep learning objectives $\ell$? A condition for answering this question is answering the following narrower and more tractable question: what are the dynamics of gradient descent near a minimiser $w$ of $\ell$ for which $\lambda_{\max}(\nabla^2\ell(w))\approx 2/\eta$? The dynamical systems theorist will immediately recognise the potential utility of invariant manifold and foliation theory. At such a minimiser, the linearisation of gradient descent is given by the matrix $I-\eta\nabla^2\ell(w)$, where $I$ is the identity on $\RB^p$; if $\lambda_{\max}(\nabla^2\ell(w))$ is simple, then with $\eta=2/\lambda_{\max}(\nabla^2\ell(w))$ this linearisation has largest singular value equal to one, with all other singular values being strictly smaller. Thus the conditions are ripe for the application, near $w$, of a \emph{sufficiently smooth centre manifold} theorem to invariantly integrate the top Hessian eigendirection, followed by a \emph{sufficiently smooth stable foliation} theorem to invariantly integrate the other eigendirections, and thereby produce a \emph{normal form} for large-step gradient descent from which the local dynamics of the algorithm can be determined. In the special case when $w$ simultaneously minimises both $\lambda_{\max}(\nabla^2\ell)$ and $\ell$, such a normal form has already been provided in a simplified setting in \cite{macdonaldeos},wherein the centre manifold is a tubular neighbourhood of the solutions near $w$ and the stable foliation is trivial.

More generally, application of classical results on centre manifolds and stable foliations to this problem are unfortunately not straightforward. Classical results typically require a uniform spectral gap hypothesis on the linearisation of the map together with a strict invariance or in/outflowing condition on the reference geometry \cite{hpsinvariant,Fenichel1974AsymptoticStability,Fenichel1977AsymptoticStabilityII}; reference geometries for deep learning objectives $\ell$, on the other hand, are often \emph{non-compact manifolds} of minimisers, with uniform spectral gaps being obtainable only over \emph{locally invariant} submanifolds which need not be either strictly invariant or in/outflowing in general\footnote{Consider for instance binary logistic regression on symmetric data: the max margin ray is a non-compact invariant manifold, with forward iterates of the map moving out along this ray to infinity. The Hessian decreasing to zero in norm along this direction \cite{salmaeos}, thus narrowing any spectral gap to zero.}. Furthermore, the large-step hypothesis that $\eta\approx \lambda_{\max}(\nabla^2\ell(w))$ allows for the gradient descent map itself to be non-invertible (even locally), which violates the invertibility hypothesis common to much of the literature \cite{hpsinvariant,Fenichel1974AsymptoticStability,Fenichel1977AsymptoticStabilityII}. The goal of this paper is therefore to provide sufficient conditions for the existence of smooth invariant manifolds and foliations in a setting where both (a) the reference geometry is merely \emph{locally invariant}, and (b) the map implementing the dynamics is not locally invertible. Since the theory of invariant manifolds and foliations is so well-developed, a thorough literature review is given below to indicate why this problem remains a gap.

\vspace{2mm}

\noindent\textbf{Invariant manifolds: } The existence of the stable and unstable manifolds through a hyperbolic fixed point \cite{hadamard,Perron1929DifferenceEquations} is one of the cornerstone results in dynamical systems theory, permitting clean theorising about the asymptotic behaviour of a system near the fixed point. This classical theory can be extended to cover pseudo-hyberbolic fixed points, yielding centre/centre-stable/centre-unstable manifolds \cite{pliss, kelley, hpsinvariant, center_regularity, deng}; it may also be extended beyond the case of a single fixed point to other reference geometries, such as the setting of a normally hyperbolic invariant manifold \cite{hpsinvariant,fenichel_singular, eldering}; the theory has in some cases also been extended to the intersection of these, yielding centre manifolds around reference geometries such as pseudo-hyperbolic invariant manifolds \cite{chow_liu_yi} or more generally invariant sets \cite{chow_liu_yi_2,bonatti_crovisier}. While much of the theory on this topic has been developed for flows and invertible maps, some of these results can also be obtained for non-invertible maps in a number of settings \cite{Chaperon2002InvariantManifolds,Chaperon2004PerronIrwin,Chaperon2008LipschitzianCore,Chaperon2008GeneratingMaps,Chaperon2015GeneratingMaps,CapinskiZgliczynski2015GeometricProof,Chen2026PartiallyNormallyHyperbolic}. The smoothness of manifolds produced in this way is well-understood, being typically guaranteed by sufficient smoothness of the map implementing the dynamics and appropriate polynomial estimates on the spectrum of its linearisation along the reference geometry. The closest of these works to applying to our setting appears to be \cite{Chen2026PartiallyNormallyHyperbolic}, which proves invariant pseudo-hyperbolic manifold theorems (including centre manifolds) in a setting including non-invertible maps along invariant sets. To the author's knowledge, the locally invariant results described here cannot obviously be derived from those of \cite{Chen2026PartiallyNormallyHyperbolic}.

\vspace{2mm}

\noindent\textbf{Invariant foliations: } In the setting of normally hyperbolic invariant manifolds, it is known that the stable and unstable manifolds are invariantly foliated \cite{hpsinvariant, Fenichel1974AsymptoticStability, Fenichel1977AsymptoticStabilityII,CapinskiZgliczynski2015GeometricProof,Chen2020InvariantGraphs} by foliations with leaves tangent to the stable and unstable directions respectively. Although the leaves of these foliations are generally as smooth as the manifolds they foliate, they are not generally \emph{transversely} smooth, being only H\"{o}lder in general \cite{hpsinvariant}. To the author's knowledge, sufficient conditions for the \emph{transverse} smoothness of these foliations in this general setting are not as well-studied, with sufficient conditions for high order transverse smoothness receiving treatment in \cite{Fenichel1974AsymptoticStability,Fenichel1977AsymptoticStabilityII,Chaperon2004PerronIrwin,deng}, and sufficient conditions for transverse $C^1$ smoothness proved in other settings in \cite{PughShubWilkinson1997Holder,Chen2020InvariantGraphs}. Transverse smoothness of invariant  horocyclic foliations has also been treated in \cite{HirschPugh1975Horocycle,Hasselblatt1994Regularity}. The closest of these works to applying to the present setting appear to be \cite{Fenichel1974AsymptoticStability,Chen2020InvariantGraphs}, which prove invariant foliation theorems with maximal smoothness in the case of a diffeomorphism with reference geometry an overflowing invariant manifold, and with $C^1$ smoothness in the case of possibly non-invertible maps with invariant reference geometry respectively. To the author's knowledge, neither of these straightforwardly imply the maximally smooth results for non-invertible maps with locally invariant reference geometries obtained in this paper.

\vspace{2mm}

\noindent\textbf{Basic definitions and statement of results: }The remainder of the introduction will recall some basic definitions and state our results. Throughout, let $M$ denote a $C^{\infty}$ manifold with tangent bundle $\T M$, $k\geq 1$ an integer, and $f:M\rightarrow M$ a $C^k$ map with tangent map $\T f:\T M\rightarrow \T M$. The restriction of $\T M$ to any subset $X\subset M$ will be denoted $\T M_X\rightarrow X$.

\begin{definition}\label{def:locallyinvariant}
\begin{enumerate}
\item A submanifold $X\subset M$ is \textbf{locally invariant} if there is a neighbourhood $U$ of $X$ in $M$ such that $f(X)\cap U\subset X$, and \textbf{diffeomorphically locally invariant} if $f|_{X\cap f^{-1}(X)}$ is a $C^k$ diffeomorphism onto its image in $X$.
\item Assuming $X\subset M$ is locally invariant, a vector subbundle $E\rightarrow X$ of $\T M_X\rightarrow X$ such that $\T f$ restricts to a map $E_{X\cap f^{-1}(X)}\rightarrow E_X$ will be called \textbf{locally invariant}.
\item Assuming $X\subset M$ is locally invariant, a foliation $\FC$ with leaves $\LC_x$ parametrised by points $x\in X$ will be called \textbf{locally invariant} if there is a neighbourhood $V$ of $X$ such that $f(\LC_x)\cap V\subset \LC_{f(x)}$ for all $x\in X\cap f^{-1}(X)$.
\end{enumerate}
\end{definition}

Having recalled these basic definitions, our results are all formulated in the following common setting: $X\subset M$ is a diffeomorphically locally invariant $C^k$-embedded submanifold; $E^{\pm}\rightarrow X$ are locally invariant $C^k$ vector subbundles of $\T M_X$ which intersect each other and $\T X$ only in the zero section and are complementary in the sense that $\T X + E^- + E^+ = \T M_X$. We will denote $g:=f|_{X\cap f^{-1}(X)}$, so that $g:\dom(g)\rightarrow X$ is a diffeomorphism onto its image, and denote $\T^{\pm}f$ for the restriction of $\T f$ to $E^{\pm}_{\dom(g)}$. Finally, we assume that $M$ admits a $C^{\infty}$ Riemannian metric with respect to which $f$ admits \emph{spectral bounds}
\begin{align}
0<\mu_1\leq \inf_{x\in \dom(g)}s_{\min}(\T_xg)&\leq \sup_{x\in \dom(g)}s_{\max}(\T_xg)\leq \mu_2 <\infty\label{eq:specboundsX}\\\sup_{x\in\dom(g)}s_{\max}(\T^-_xf)\leq \nu_1 &< \nu_2\leq \inf_{x\in\dom(g)}s_{\min}(\T^+_xf)\label{eq:specboundsE}
\end{align}
where $s_{\min}$ and $s_{\max}$ denote smallest and largest singular value respectively. We make no assumption that $\T^-f$ is invertible.

\begin{theorem}[Locally invariant manifold theorems]\label{thm:invariantmanifolds}
For each $1\leq j\leq k$, define
\begin{align}
\rho^{\pm}_j:=\begin{cases} \nu_2^{-1}\max\{\mu_2,\nu_1\}^j&\text{ if $\pm=-$}\\\nu_1\min\{\mu_1,\nu_2\}^{-j}&\text{ if $\pm=+$}.\end{cases}\label{eq:rhoj}
\end{align}
Then if $\rho^{\pm}_j<1$ for all $1\leq j\leq k$, for any relatively open, relatively pre-compact set $A^-\subset \dom(g)$ with $A^+:=g(A^-)$, there exists a $C^k$, locally invariant manifold $N^{\pm}$ containing $A^{\pm}$ such that $\T N^{\pm}_{A^{\pm}} = \T A^{\pm}+E^{\pm}_{A^{\pm}}$.
\end{theorem}

Note that the hypothesis $\rho^{\pm}_j<1$ for all $j$ does not necessarily require $\nu_1<1$ (when $\pm=-$) or $\nu_2>1$ (when $\pm=+$). It thus permits the construction of centre(-stable/unstable) manifolds in certain cases, and as such we will sometimes refer to these manifolds as merely \emph{relatively} stable/unstable. For the invariant foliation theorems, we must assume in addition that $f$ and $E^{\pm}\rightarrow X$ are all $C^{k,1}$.

\begin{theorem}[Locally invariant foliation theorems]\label{thm:invariantfoliations}
Assume that $f$ and $E^{\pm}\rightarrow X$ are $C^{k,1}$. For all $0\leq j\leq k$, define
\begin{align}
\gamma^{\pm}_j:=\begin{cases} \nu_1\,\mu_2^j\,\min\{\mu_1,\nu_2\}^{-1}&\text{ if $\pm=-$}\\ \nu_2^{-1}\,\mu_1^{-j}\,\max\{\mu_2,\nu_1\}&\text{ if $\pm=+$}\end{cases}.\label{eq:gammaj}
\end{align}
Then if $\gamma^{\pm}_j<1$ for all $0\leq j\leq k$, for any relatively open, relatively pre-compact $A^-\subset \dom(g)$ with $A^+:=g(A^-)$, there is a $C^k$, locally invariant manifold $N^{\pm}$ containing $A^{\pm}$ admitting a $C^k$, locally invariant foliation with leaves which are parametrised by $A^{\pm}$ and tangent at $A^{\pm}$ to $E^{\pm}_{A^{\pm}}$.
\end{theorem}

Note that the condition $\gamma^{\pm}_1<1$ \emph{necessitates} $\nu_1<1$ (when $\pm=-$) or $\nu_2>1$ (when $\pm=+$); thus the foliations produced by our theorems are genuinely stable/unstable, not merely relatively so.

It is clear from our proof that, in the relatively unstable case $(\pm=+$), the $N^+$ produced by our Theorem \ref{thm:invariantmanifolds} is in fact \emph{diffeomorphically} locally invariant manifold; extending $E^-$ to an invariant complement of $\T N^+$ using Lemma \ref{lem:bundleextension}, therefore, this $N^+$ may therefore be fed as a hypothesis into Theorem \ref{thm:invariantfoliations} to produce the following result, which we expect will be of great utility in the theory of large-step optimisation.

\begin{corollary}[Synthesis of Theorems \ref{thm:invariantmanifolds} and \ref{thm:invariantfoliations}]\label{cor:mfldandfoln}
Assume that $k\geq 2$, that $f$ and $E^{\pm}\rightarrow X$ are $C^{k,1}$, and that $\rho^{+}_j<1$ for all $1\leq j\leq k$. Suppose also that $\nu_2':=\sup_{x\in\dom(g)}s_{\max}(\T_x^+f)<\infty$, and that
\begin{align}\label{eq:mfldandfolnestimate}
\nu_1\,\max\{\mu_2,\nu_2'\}^j\,\min\{\mu_1,\nu_2\}^{-1}<1
\end{align}
for all $0\leq j\leq k$. Then for any relatively open, relatively pre-compact $A\subset \mathrm{im}(g)\cap\mathrm{dom}(g)$, there is a $C^k$, diffeomorphically locally invariant manifold $N^+$ containing $A$ with $\T N^+_A = \T A+E^+_A$ whose points parametrise the leaves of a $C^{k-1}$, locally invariant foliation of a neighbourhood of $N^+$ with leaves tangent to $E^-_A$ at $A$.
\end{corollary}

Corollary \ref{cor:mfldandfoln} follows immediately upon applying Theorem \ref{thm:invariantfoliations} to the diffeomorphically locally invariant $N^+$ (shrunk if necessary) output by Theorem \ref{thm:invariantmanifolds}; the complement $E^-$ is extended from $A$ across $N^+$ using Lemma \ref{lem:bundleextension}, and the quantity on the left in \eqref{eq:mfldandfolnestimate} is the form taken by $\gamma^-_j$ when $X$ is replaced by $N^+$. The requirement that $k\geq 2$ and the differentiability drop from $C^k$ to $C^{k-1}$ between the manifold and the foliation in Corollary \ref{cor:mfldandfoln} are, the author suspects, not essential. They result from the smoothness hypothesis on $E^{\pm}\rightarrow X$ which seems unlikely to be necessary, cf. \cite[Section 4]{hpsinvariant}. Removing these suboptimal hypotheses is left for future work.

Before proceeding with an outline of the paper, a remark on proof technique. Both Theorems \ref{thm:invariantmanifolds} and \ref{thm:invariantfoliations} are proved using graph transform arguments whose elements are far from novel. In particular, the pseudo-hyperbolic setting we consider as well as the $*$-norms with which we equip our section spaces are due to Hirsch, Pugh and Shub \cite[Section 5]{hpsinvariant}; our use of cutoff functions in extending our geometric graph transforms to analytic graph transforms to which the Banach fixed point theorem can be applied is due to Bonatti and Crovisier \cite{bonatti_crovisier}; our proof of smoothness of the relevant objects follows that of Sandstede and Theerakarn \cite{center_regularity}. The only technical novelty, to the extent that any can be claimed by the author, consists in the following three points concerning the synthesis of these elements:
\begin{enumerate}
  \item Since $f$ is not assumed to be a diffeomorphism, the author has spent a great deal of effort developing the arguments such that the (relatively) stable/unstable cases follow from the one formally identical proof, with differences between the two cases reducing to the permutation of the symbols $\pm$ (see especially Propositions \ref{prop:geometricgraphtransforms}, \ref{prop:geometricgraphtransformsfoliation}). To the author's knowledge, this argument has not yet appeared in the literature. However, a similarly flexible (yet apparently distinct) argument using correspondences was invented by Champeron \cite{Chaperon2002InvariantManifolds,Chaperon2004PerronIrwin,Chaperon2008LipschitzianCore,Chaperon2008GeneratingMaps,Chaperon2015GeneratingMaps} and further developed by Chen \cite{Chen2020InvariantGraphs,Chen2026PartiallyNormallyHyperbolic}.
  \item Our foliations, parametrised by locally-defined sections of $\T X\oplus E^-\oplus E^+\rightarrow E^{\pm}$, have leaves given by the exponentiated images of the fibres of $E^{\pm}$ under such sections. Our assumptions on the relevant sections are sufficiently strong that these leaves do not overlap by construction, whereas ordinarily a dynamical argument is used to ensure that leaves are non-overlapping, cf. \cite[Corollary 5.6]{hpsinvariant}. This dynamical argument, requiring arbitrarily many iterates, is not available in our locally invariant setting, thus necessitating a different construction; see the paragraph following Proposition \ref{prop:foliationsectionscomplete}. While this proof is suitable for the existence of foliations with at least $C^{0,1}$ transverse regularity, the author does not believe that it extends to give existence of transversely continuous foliations as considered in \cite{hpsinvariant,Fenichel1974AsymptoticStability}.
  \item In order to prove our theorems over merely locally invariant bases, we extended the idea, due to Bonatti and Crovisier in \cite{bonatti_crovisier}, of modifying the graph transform by a cutoff function. Beyond the radial cutoffs used in \cite{bonatti_crovisier}, the cutoffs we use also cut-off in the base direction to account for the mere local invariance of the base.
\end{enumerate}

\noindent\textbf{Structure of paper: }Section \ref{sec:setup} introduces notation and fleshes out the common framework in which our invariant manifold and foliation theorems will be proved. Following this, Section \ref{sec:invariantmanifolds} proves Theorem \ref{thm:invariantmanifolds}, and Section \ref{sec:invariantfoliations} proves Theorem \ref{thm:invariantfoliations}. Section \ref{sec:technical} contains several technical lemmas needed for the proofs. Section \ref{sec:AIuse} contains the AI use statement.

\section{Setup}\label{sec:setup}

\subsection{Notation}

The closure of a subset $A$ of a topological space $X$ will be denoted $\overline{A}$. Given vector bundle $E\rightarrow X$ admitting a continuous fibrewise norm $|\cdot|$ and $R>0$, denote
\[
E(R):=\{e_x\in E:|e_x|<R\}.
\]
The direct sum of such bundles will be equipped with the $L^{\infty}$ norm across factors, so that
\[
(E^1\oplus E^2)(R) = \{(e^1_x,e^2_x)\in E^1\oplus E^2:\max\{|e^1_x|,|e^2_x|\}<R\}.
\]
The image of the zero section in a vector bundle $E\rightarrow X$ over a subset $A\subset X$ will be denoted $0_A$. The restriction of any fibre bundle $E\rightarrow X$ to a subset $A\subset X$ will be denoted $E_A$. Given a subset $A$ of a metric space $(X,d)$ and $R>0$, we denote
\[
A(R):=\{x\in X:d(x,A)<r\}
\]
for the $R$-tube of $A$. Given a $C^{\infty}$ vector bundle $E\rightarrow X$ over a $C^{\infty}$ Riemannian manifold $X$, a connection for $E$, $\nabla\sigma$ will denote the covariant derivative of a differentiable section $\sigma$ and is a section of $E\otimes \T^*X\rightarrow X$. Given an integer $j\geq 2$, the $j^{th}$ covariant derivative of a sufficiently differentiable section $\sigma$ of $E\rightarrow X$ will be denoted $\nabla^j\sigma$, thought of as a section of $E\otimes(\T^*X)^{\otimes j}$ and defined by
\[
\nabla^{j}\sigma:=(\nabla\otimes\nabla_{LC}^{\otimes j-1})\nabla^{j-1}\sigma,
\]
where $\nabla_{LC}$ denotes the Levi-Civita connection for $X$ acting on sections of the cotangent bundle.

\subsection{Common framework}

In this subsection, we establish some tools common to our proofs of both our invariant manifold and our invariant foliation theorems, working always under the setting and assumptions of the Introduction.

To construct invariant manifolds and foliations, we equip $E^{\pm}\rightarrow X$, which a-priori inherit geometry of only limited regularity from $M$, with $C^{\infty}$ geometries. First, recall the standard fact \cite[Theorem 2.2.9 p.51, Theorem 4.3.5 p.101]{hirsch} that the $C^k$ bundles $E^{\pm}$ over the $C^k$ manifold $X$ admit compatible $C^{\infty}$ structures. Equipped with these $C^{\infty}$ structures, fix any $C^{\infty}$ Riemannian metric on $X$ and $C^{\infty}$ Euclidean structures and compatible connections on $E^{\pm}\rightarrow X$. With respect to these $C^{\infty}$ structures, denote the Riemannian exponential map for $X$ by $\exp^X$, and for $x,y\in X$ sufficiently close to be connected by a unique length-minimising geodesic in $X$ denote
\[
P^{\pm}_{x\rightarrow y}:E^{\pm}_x\rightarrow E^{\pm}_y
\]
for the corresponding parallel transport map along this geodesic. Furthermore, denote
\[
E:=E^-\oplus E^+,\qquad P_{x\rightarrow y}:=P^-_{x\rightarrow y}\oplus P^+_{x\rightarrow y}:E_x\rightarrow E_y
\]
for the direct sum bundle and parallel transport.

Note then that if $A\subset g^{-1}(X)\cap X$ is a relatively open, relatively pre-compact subset, then for any $\mu_1'<\mu_1\leq \mu_2<\mu_2'$ and $\nu_1<\nu_1'<\nu_2'<\nu_2$, the $C^{\infty}$ metrics on $E^{\pm},X$ may be chosen sufficiently $C^0$-close to those induced from $M$ over $f(\overline{A})\cup\overline{A}$ so that the spectral bounds
\[
\mu_1'\leq \inf_{x\in A}s_{\min}(\T_xg)\leq \sup_{x\in  A}s_{\max}(\T_xg)\leq \mu_2',\qquad \sup_{x\in A}(\T^-_xf)\leq \nu_1'<\nu_2'\leq\inf_{x\in A}(\T^+_xf)
\]
hold with respect to these $C^{\infty}$ metrics. Therefore, to avoid notational clutter, in what follows we assume without loss of generality that the spectral bounds \eqref{eq:specboundsX}, \eqref{eq:specboundsE} hold with respect to these $C^{\infty}$ auxiliary metrics.

At a high level, our framework for constructing invariant manifolds and foliations replaces $M$ by a fibre bundle over $X$ consisting of exponential charts, and replaces $f$ by a morphism of this fibre bundle. The fibre bundles with which we will be concerned are subbundles of
\[
\MC:=\T X\oplus E.
\]
The bundle $\MC\rightarrow X$ should be thought of as a bundle of Riemannian charts for $M$ about $X$ in the following sense. For any $x\in X$, let $R^{\T}_x>0$ denote the convexity radius about $x$ and $B_X(x,R^{\T}_x)$ the (strongly convex) Riemannian ball about $x$; then $\exp^X_x:\T_x X(R^{\T}_x)\rightarrow B_X(x,R^{\T}_x)$ is a $C^{\infty}$ diffeomorphism, thus giving $C^{\infty}$ diffeomorphisms $\theta^{\pm}_x:\T_x X(R^{\T}_x)\times E^{\pm}_x\rightarrow E^{\pm}_{B_X(x,R^{\T}_x)}$ defined by
\[
\theta^{\pm}_x(\xi_x,w_x):=P_{x\rightarrow\exp^X_x(\xi_x)}w_x,\qquad (\xi_x,w_x)\in \T_x X(R^{\T}_x)\times E^{\pm}_x,
\]
and similarly $\theta_x:\T_x X(R^{\T}_x)\times E_x\rightarrow E_{B_X(x,R^{\T}_x)}$. Denoting by $\exp^M$ the Riemannian exponential map for $M$, for any $x\in X$, there is $R>0$ such that the map $\Theta_x:\MC(R_x)_x\rightarrow M$ defined by
\[
\Theta_x(\xi_x,u_x,v_x):=\exp^M_{\exp^X_x(\xi_x)}\big(\theta^-_x(\xi_x,u_x) + \theta^+_x(\xi_x,v_x)\big)
\]
is a $C^k$ diffeomorphism onto its image. Moreover, for any pre-compact set $K\subset X$, $R_K>0$ may be chosen so that the $\Theta_x:\MC(R_K)_x\rightarrow M$ are $C^k$ diffeomorphisms onto their images uniformly over all $x\in K$. In particular, for any relatively open, relatively pre-compact $A\subset f^{-1}(X)\cap X$, $R_A>0$ may be chosen such that the formula
\[
F(\xi_x,u_x,v_x):=\Theta_{f(x)}^{-1}\circ f\circ \Theta_x(\xi_x,u_x,v_x),\qquad (\xi_x,u_x,v_x)\in \MC(R_A)_A
\]
defines a $C^k$ bundle-morphism $\MC(R_A)_A\rightarrow \MC$.

The Levi-Civita connection on $X$ and the metric connections on $E^{\pm}\rightarrow X$ furnish $C^{\infty}$ identifications
\begin{equation}\label{eq:connection}
\T_{(\xi_x,u_x,v_x)}\MC\cong \T_xX\oplus \MC_x,\qquad \forall (\xi_x,u_x,v_x)\in \MC,
\end{equation}
with respect to which the tangent map $\T F:\T\MC(R_A)_A\rightarrow \T\MC$ can be written
\[
\D F = \big(\D^i_jF\big)_{i,j\in\{0,\T,-,+\}} = \begin{pmatrix} \D^0_0F & \D^0_{\T}F & \D^0_-F & \D^0_+F \\ \D^{\T}_0F & \D^{\T}_{\T}F & \D^{\T}_-F & \D^{\T}_+F\\ \D^-_0F & \D^-_{\T}F & \D^-_{-}F & \D^-_+F \\ \D^+_0F & \D^+_{\T}F & \D^+_-F & \D^+_+F\end{pmatrix}.
\]
Note crucially that the identification \eqref{eq:connection} coincides over the zero-section with that induced by the derivative of the zero-section; consequently,
\[
\D F_{(\xi_x,u_x,v_x)} = \T_xg\oplus \T_xg\oplus \T^-_xf\oplus \T^+_xf + o_R(1)
\]
as maps $\T_xX\oplus\MC_x\rightarrow \T_{f(x)}X\oplus\MC_{f(x)}$ uniformly over $(\xi_x,u_x,v_x)\in \MC(R)_A$ as $R\rightarrow 0$. Finally, denote $F^{\T,\pm}$ and $F^{\pm}$ for the projections of $F$ onto $\T X\oplus E^{\pm}$ and $E^{\pm}$ respectively.

\section{Locally invariant manifold theorems}\label{sec:invariantmanifolds}

In this section, we make the blanket assumption that $k\in\NB$ satisfies $k\geq 1$. The invariant manifold theorems will both be derived from the following common construction, assuming $C^{\infty}$ metrics on $E^{\pm},X$ as above. Given any relatively open, relatively pre-compact subset $C\subset X$ and any sufficiently small $R>0$, observe that any manifold containing $C$ and tangent along $C$ to $\T X_C+E^{\pm}_C$ can locally be written as the exponentiated image of a section of the bundle $E(R)_C\rightarrow E^{\pm}(R)_C$ which vanishes to first order on the zero-section.

Defining
\[
Y^{\pm}:=\T X\oplus E^{\pm},
\]
such manifolds also admit a parametrisation by certain local sections of $\MC\rightarrow Y^{\pm}$, as summarised in the following lemma.

\begin{lemma}\label{lem:equivariant}
Let $A\subset X$ be relatively open, relatively pre-compact and let $R>0$ be such that $B_X(x,R)$ is strongly convex for all $x\in A$. Given $x,x'\in A$ such that $B_X(x,R)\cap B_X(x',R)\neq \emptyset$, define $\theta^{\pm}_{x'x}:=(\theta_{x'}^{\pm})^{-1}\circ\theta_x^{\pm}:\log^X_x(B_X(x,R)\cap B_X(x',R))\times E^{\pm}_x\rightarrow \log^X_{x'}(B_X(x,R)\cap B_X(x',R))\times E^{\pm}_{x'}$. Then the identity
\[
\tsig|_{\T_xX(R)\oplus E^{\pm}_x} = \theta_x^{-1}\circ\sigma\circ\theta_x^{\pm}
\]
defines a bijective correspondence between sections $\sigma$ of $E(R)_C\rightarrow E^{\pm}(R)_C$ and sections $\tsig$ of $\MC(R)_A\rightarrow \T X(R)_A\oplus E^{\pm}(R)_A$ which are equivariant under the change of chart maps in the sense that
\[
\theta_{x'x}\circ\tsig|_{\T_x X(R)\oplus E^{\pm}_x} = \tsig|_{\T X_{x'}(R)\oplus E^{\pm}_{x'}}\circ \theta^{\pm}_{x'x}
\]
on $\log^X_x(B_X(x,R)\cap B_X(x',R))\times E^{\pm}_x$.
\end{lemma}

Given $C$ and $R$ as above, and numbers $\Lb\in\RB^k_{\geq 0}$, denote by $S^{\pm}(C,R,\Lb)$ the space of $C^{k-1,1}$ sections of $E(R)_C\rightarrow E^{\pm}(R)_C$ which are differentiable at every point of the zero section in $E^{\pm}(R)_C$, satisfying
\[
\sigma(0_x) = 0_x,\qquad \nabla\sigma(0_x) = 0_x,\qquad \forall x\in C,
\]
admitting covariant derivative bounds
\[
\|\nabla^j\sigma\|_{\infty}\leq L_j\quad\forall j=1,\dots,k,
\]
with the $\|\nabla^k\sigma\|_{\infty}$ defined, via Rademacher's theorem, by the essential supremum, and for which the norm
\[
\|\sigma\|_*:=\sup_{u_x\in E^{\pm}(R)_C\setminus 0_C}\frac{|\sigma(u_x)|_x}{|u_x|_x}
\]
is finite. Note that by Lemma \ref{lem:lipschitzvsderivative}, any element of $S^{\pm}(C,R,\Lb)$ has $\nabla^{k-1}\sigma$ being $L_k$-Lipschitz.

\begin{proposition}\label{prop:manifoldsectionscomplete}
For any relatively open, relatively pre-compact $C\subset X$, any $R>0$ and any $\Lb\in\RB^k_{\geq 0}$, $(S^{\pm}(C,R,\Lb),\|\cdot\|_*)$ is a complete metric space.
\end{proposition}

\begin{proof}
As in \cite[Sublemma 5.2]{hpsinvariant}, a sequence $\{\sigma_n\}_{n=1}^{\infty}\subset S^{\pm}(C,R,\Lb)$ is Cauchy in $\|\cdot\|_{*}$ only if it is Cauchy in $\|\cdot\|_{\infty}$, only if it converges to some continuous section $\sigma$. This $\sigma$ itself satisfies $\|\sigma\|_{*}<\infty$ by the Cauchy property in $\|\cdot\|_*$. Indeed, taking $N$ sufficiently large that for all $n,m\geq N$ one has $\|\sigma_n-\sigma_m\|\leq 1$, one sees that for all $m\geq N$ one has
\[
\|\sigma_m\|_*\leq \|\sigma_N\|_*+1
\]
from which it follows that $|\sigma(u_x)|/|u_x|\leq \|\sigma_N\|+1$ for all nonzero $u_x\in E^{\pm}$, giving finiteness of $\|\sigma\|_*$. That $\sigma_n\rightarrow \sigma$ in $\|\cdot\|_*$ then follows from a similar argument to that of \cite[Sublemma 5.2]{hpsinvariant}.

We now demonstrate the differentiability of $\sigma$ at any point $0_{x_0}$ in the zero section. For all $u_x$ sufficiently close to $0_{x_0}$, one has the estimates
\[
\frac{|\sigma(u_x)|}{d(u_x,0_{x_0})}\leq \frac{|\sigma_n(u_x)|}{d(u_x,0_{x_0})} + \frac{|\sigma_n(u_x)-\sigma(u_x)|}{d(u_x,0_{x_0})}\leq \frac{|\sigma_n(u_x)|}{|u_x|} + \|\sigma_n-\sigma\|_*
\]
where $d(u_x,0_{x_0})\geq |u_x|$ is the Sasaki metric. Since the $\sigma_n$ are all differentiable at $0_{x_0}$ with $\nabla\sigma_n(0_{x_0}) = 0_{x_0}$ and $\|\sigma_n-\sigma\|_*\rightarrow 0$, it follows that $\lim_{u_x\rightarrow 0_{x_0}}|\sigma(u_x)|/d(u_x,0_{x_0}) = 0$, hence that $\sigma$ is differentiable at $0_{x_0}$ with $\nabla\sigma(0_{x_0}) = 0_{x_0}$.

For the derivative bounds and higher differentiability, it suffices to demonstrate the result locally. Fixing any point $u_x\in E^{\pm}(R)$, let $\overline{B}$ be a closed, geodesically strongly convex ball with $u_x$ contained in its interior $B$. Since $\|\cdot\|_*$ convergence implies uniform convergence, $\sigma_n$ converges uniformly to $\sigma$ over $\overline{B}$. Moreover, for each $j=1,\dots,k-1$, the uniform bounds $\|\nabla^j\sigma_n\|_{\infty}\leq L_j$ and $\Lip(\nabla^{k-1}\sigma_n)\leq L_k$ imply that the family $\{\nabla^j\sigma_n\}_{n=1}^{\infty}$ is uniformly bounded and uniformly equicontinuous, so the limit $\sigma$ of $\sigma_n$ is $C^{k-1}$ by Lemma \ref{lem:vbaa}. Furthermore, $\nabla^{k-1}\sigma$, being the uniform limit of the Lipschitz sections $\nabla^{k-1}\sigma_n$, is itself Lipschitz with $\|\nabla^k\sigma\|_{\infty} = \mathrm{Lip}(\nabla^{k-1}\sigma)\leq L_k$ by Lemma \ref{lem:lipschitzvsderivative}.
\end{proof}

For narrational convenience, we will refer to elements of this space as \emph{admissible sections}, whose graphs we shall refer to as \emph{admissible graphs}. The next proposition is the key geometric construction enabling the definition of the graph transform. At a high level, it says that any admissible graph from $S^{\pm}$ is \emph{mapped to} by some admissible graph under $F$ (when $\pm=-$) or \emph{maps to} some admissible graph under $F$ (when $\pm=-$). We formalise this pair of notions as follows. 


Let $A^-\subset\mathrm{dom}(g)$ be relatively open and relatively pre-compact and denote $A^+:=g(A^-)$. Denote points in $A^{\pm}$ by $x^{\pm}$. Let $R^{\pm}>0$ be sufficiently small that $F:\MC(R^-)_{A^-}\rightarrow \MC(R^+)_{A^+}$ is defined; this $F$ is the local expression of $f$ in the coordinates $\Theta$. Set $A:=A^-\cup A^+$, $R:=\min\{R^-,R^+\}$ and $B:=A(R)$. If $\sigma$ is an admissible section of $E(R)_{B}\rightarrow E^{\pm}(R)_{B}$ with lift $\tsig$ to $\MC(R)_{A}\rightarrow Y^{\pm}(R)_{A}$, we seek $0<r^{\pm}\leq R$ and section $\sigma^{\pm}$ of $E_{A^{\pm}(r^{\pm})}\rightarrow E^{\pm}_{A^{\pm}(r^{\pm})}$ such that either (i) when $\pm=-$, the image of $\tsig$ contains the image of $F\circ\tsig^-$, or (ii) when $\pm=+$, the image of $F\circ\tsig$ contains the image of $\tsig^+$. In the former case, it suffices to assign to each $(\xi_{x^-},u_{x^-})\in Y^{-}_{A^-(r^-)}$ a unique $v_{x^-}'\in E^+_{A^-(R^-)}$ and $(\xi'_{x^+},u'_{x^+})\in Y^-_{A^+(R^+)}$ such that
\[
F(\xi_{x^-},u_{x^-},v'_{x^-}) = \tsig(\xi'_{x^+},u'_{x^+});
\]
the $v'_{x^-}$ would then be the fibre value taken by $\tsig^-$ at $(\xi_{x^-},u_{x^-})$, while $(\xi'_{x^+},u'_{x^+})$ would be the base value obtained by evaluating $F^{0-}$ on $(\xi_{x^-},u_{x^-},v'_{x^-})$. In the latter case, for each $(\xi_{x^+},v_{x^+})\in Y^+_{A^+(r^+)}$, it suffices to determine $u'_{x^+}\in E^-_{A^+(R^+)}$ and $(\xi'_{x^-},v'_{x^-})\in E^+_{A^-(R^-)}$ such that
\[
F\circ\tsig(\xi'_{x^-},v'_{x^-}) = (\xi_{x^+},u'_{x^+},v_{x^+});
\]
the $u'_{x^+}$ would then be the fibre value taken by $\tsig^+$ at $(\xi_{x^+},v_{x^+})$, while $(\xi'_{x^-},v'_{x^-})$ would be the base value mapping to $(\xi_{x^+},v_{x^+})$ under $F^{0+}\circ\tsig$.

Denoting $g_+:=g$ and $g_-:=g^{-1}$ and defining
\[
W^{\pm}:=\begin{cases} g_+^*\T X\oplus g_+^*E^-\oplus E^+&\text{ if $\pm=-$}\\ g_-^*\T X\oplus E^-\oplus g_-^*E^+&\text{ if $\pm=+$}\end{cases},\qquad Z^{\pm}:=Y^{\pm}\oplus W^{\pm}
\]
and $\Phi^{\pm}_{\sigma}:Z^{\pm}(R)_{A^{\pm}}\rightarrow \MC_{A^+}$ by
\begin{align}\label{eq:Phi^-}
\Phi^{-}_{\sigma}(\xi_{x^-},u_{x^-},\xi'_{x^+},u'_{x^+},v'_{x^-}):=F(\xi_{x^-},u_{x^-},v'_{x^-})-\tsig(\xi'_{x^+},u'_{x^+})
\end{align}
and
\begin{align}\label{eq:Phi^+}
\Phi^+_{\sigma}(\xi_{x^+},v_{x^+},\xi'_{x^-},u'_{x^+},v'_{x^-}):=F\circ\tsig(\xi'_{x^-},v'_{x^-}) - (\xi_{x^+},u'_{x^+},v_{x^+}),
\end{align}
one sees that it suffices to find $0<r^{\pm}\leq R^{\pm}$ so that the zero set of $\Phi^{\pm}_{\sigma}|_{Z^{\pm}(r^{\pm},R^{\pm})}$ is the graph of some section $\tau^{\pm}$ of $Z^{\pm}(r^{\pm},R^{\pm})_{A^{\pm}}\rightarrow Y^{\pm}(r^{\pm})_{A^{\pm}}$. Denoting $\pi^{\pm}_{Z,\MC}:Z^{\pm}\rightarrow\MC$ for the obvious projection, $\tsig^{\pm}:=\pi^{\pm}_{Z,\MC}\circ\tau^{\pm}$ is equivariant under change of chart maps because $F$ and $\tsig$ are. Moreover, denoting
\[
\phi^{\pm}_{\sigma}(\tilde{\eta}):=\begin{cases} F\circ\tilde{\eta} - \tsig\circ F^{\T,-}\circ\tilde{\eta} &\text{ if $\pm=-$}\\ F\circ\tsig - \tilde{\eta}\circ F^{\T,+}\circ\tsig&\text{ if $\pm=+$},\end{cases}
\]
then $\tsig^{\pm}$ satisfies the familiar graph transform identity $\phi^{\pm}_{\sigma}(\tsig^{\pm})\equiv 0$. The following proposition demonstrates that these $\sigma^{\pm}$ can always be found.

\begin{proposition}[Geometric graph transforms]\label{prop:geometricgraphtransforms}
Let $A^-\subset f^{-1}(X)\cap X=\operatorname{dom}(g)$ be relatively open, relatively pre-compact, define $A^+:=g(A^-)$ and $A:=A^-\cup A^+$, and for each $j=1,\dots,k$, define
\[
\rho^{\pm}_j:=\begin{cases} \nu_2^{-1}\max\{\nu_1,\mu_2\}^j&\text{ if $\pm=-$}\\\nu_1\min\{\nu_2,\mu_1\}^{-j}&\text{ if $\pm=+$}\end{cases}
\]
Let $\mathbf{L} = (L_1,\dots,L_k)\in \RB^{k}_{\geq 0}$ with $\max\{\nu_2^{-1},\nu_1\}L_1<1$. Then there are positive functions $r^{\pm}(R) = \Theta(R)\leq R$ and non-negative functions $\delta^{\pm}_j(R) = o_R(1)$, $\epsilon^{\pm}_1(R)= o_R(1)$ and $\epsilon^{\pm}_j(R,\Lb_{1:j-1}) = O_R(1)$ for all $j=2,\dots,k$ defined for all sufficiently small $R>0$ such that, defining
\[
\Lb^{\pm}(R):=\big((\rho^{\pm}_j + \delta^{\pm}_j(R))L_j + \epsilon^{\pm}_j(R,\Lb_{1:j-1})\big)_{j=1}^k,
\]
every $\sigma\in S^{\pm}(A(R),R,\Lb)$ admits a unique $\sigma^{\pm}\in S^{\pm}(A^{\pm}(r^{\pm}),r^{\pm},\Lb^{\pm}(R))$ such that $\phi^{\pm}_{\sigma}(\tsig^{\pm})\equiv 0$. Furthermore, there is
\[
\kappa^{\pm}(R,L_1):=\begin{cases} \nu_1\nu_2^{-1}(1-\nu_2^{-1}L_1)^{-1} + o_R(1)&\text{ if $\pm=-$}\\ \nu_1\nu_2^{-1}(1-\nu_1L_1)^{-1} +o_R(1)&\text{ if $\pm=+$}\end{cases},
\]
such that
\[
\|\sigma_1^{\pm}-\sigma_2^{\pm}\|_*\leq \kappa^{\pm}(R,L_1)\|\sigma_1-\sigma_2\|_*
\]
for all $\sigma_1,\sigma_2\in S^{\pm}(A(R),R,\Lb)$.
\end{proposition}

\begin{proof}
We break the proof into (1) existence and uniqueness of $\sigma^{\pm}$ given $\sigma$; (2) estimating the derivative bounds; (3) proving the contraction estimate.

\noindent{\textbf{Existence and uniqueness of $\sigma^{\pm}$:}} Fixing small $R>0$ and $\sigma\in S^{\pm}(A(R),R,\Lb)$, let $\Phi^{\pm}_{\sigma}:\overline{Z^{\pm}(R)}_{A^{\pm}}\rightarrow \MC_{A^+}$ be defined as in \eqref{eq:Phi^-} and \eqref{eq:Phi^+}. Denote the derivative of $\Phi^{\pm}_{\sigma}$ with respect to the final three variables, expressed with respect to the identifications supplied by the connections, by $\V\Phi^{\pm}_{\sigma}$. Since $\sigma$ is at least Lipschitz, this derivative is defined almost everywhere by Rademacher's theorem, and takes the form
\begin{align}\label{eq:VPhi}
\V\Phi^-_{\sigma} = \begin{pmatrix} -I_{\T X} & 0 & \D^{\T}_+F \\ 0 & -I_{E^-} & \D^-_+F \\ -\nabla_{\T}\tsig & -\nabla_{-}\tsig & \D^+_+F\end{pmatrix},\qquad \V\Phi^+_{\sigma} = \begin{pmatrix} \D^{\T}_{\T}F +\D^{\T}_-F\,\nabla_{\T}\tsig & 0 & \D^{\T}_+F + \D^{\T}_-F\,\nabla_+\tsig \\ \D^-_{\T}F + \D^-_-F\,\nabla_{\T}\tsig & -I_{E^-} & \D^-_+F + \D^-_-F\,\nabla_+\tsig \\ \D^+_{\T}F + \D^+_{-}F\,\nabla_{\T}\tsig & 0 & \D^+_+F + \D^+_-F\,\nabla_+\tsig\end{pmatrix}
\end{align}
at any point where defined. Using compactness of $\overline{A}\subset A(R)$, choose a $C^k$ vector bundle morphism $L^{\pm}:W^{\pm}\rightarrow g_+^*\MC$ which is uniformly within $R$ of $(-I_{\T X})\oplus(-I_{E^-})\oplus\T^+f$ (when $\pm=-$) or $\T g\oplus (-I_{E^-})\oplus\T^+f$ (when $\pm=+$) over $A$ in the fibrewise operator norm. Then define the fibre bundle morphism $\Psi^{\pm}_{\sigma}:\overline{Z^{\pm}(R)}_A\rightarrow W^{\pm}_A$ by
\[
\Psi^{\pm}_{\sigma}(y^{\pm},w^{\pm}) = w^{\pm}-(L^{\pm})^{-1}\Phi^{\pm}_{\sigma}(y^{\pm},w^{\pm}),\qquad (y^{\pm},w^{\pm})\in Z^{\pm}(R)_A,
\]
where $y^{\pm}\in Y^{\pm}(R)_A$ and $w^{\pm}\in Z^{\pm}(R)_{y^{\pm}}$. The zero set of $\Phi^{\pm}_{\sigma}$ is then the fixed set of $\Psi^{\pm}_{\sigma}$, which we now prove to be the graph of a (necessarily unique) section. To demonstrate the existence of such a fixed section we apply the vector bundle Banach fixed point theorem (Lemma \ref{lem:vbbanach}).

To apply this theorem, we must check that $\Psi^{\pm}_{\sigma}$ is uniformly vertically contractive and preserves a fibrewise tube. For the former, observe that at any point where $\V\Psi^{\pm}_{\sigma}$ is differentiable, \eqref{eq:VPhi} and the definition of $L^{\pm}$ imply that
\[
|\V\Psi^{\pm}_{\sigma}|\leq |(L^{\pm})^{-1}||L^{\pm}-\V\Phi^{\pm}_{\sigma}| \leq q^{\pm}(R,L_1) = \begin{cases} \nu_2^{-1}L_1 + o_R(1)&\text{ if $\pm=-$}\\ \nu_1L_1 + o_R(1)&\text{ if $\pm=+$}\end{cases}
\]
uniformly over $Z^{\pm}(R)_A$ and $\sigma\in S^{\pm}(A(R),R,\mathbf{L})$. Shrinking $R$ if necessary and invoking the hypothesis on $L_1$ gives $q^{\pm}(R,L_1)<1$, making $\Psi^{\pm}_{\sigma}$ a fibrewise contraction. For the latter, note that
\[
|\Psi^{-}_{\sigma}(\xi_{x^-},u_{x^-},0_{W^-_{x^-}})| = |(L_{x^-}^-)^{-1}F(\xi_{x^-},u_{x^-},0_{x^-})|\leq a^-(R)\max\{|\xi_{x^-}|,|u_{x^-}|\}
\]
and
\[
|\Psi^+_{\sigma}(\xi_{x^+},v_{x^+},0_{W^+_{x^+}})| = |(L^+_{x^+})^{-1}(\xi_{x^+},0,v_{x^+})|\leq a^+(R)\max\{|\xi_{x^+}|,|v_{x^+}|\},
\]
where
\[
a^-(R) = \max\{\mu_2,\nu_1\} + o_R(1),\qquad a^+(R) = \min\{\mu_1,\nu_2\}^{-1} + o_R(1).
\]
Fixing $\alpha\in(0,1)$ and defining
\[
r^{\pm}(R,\delta):=\min\bigg\{R,\alpha\frac{1-q^{\pm}(R,L_1)}{a^{\pm}(R)}R\bigg\}\leq R,
\]
shrinking $R$ yet further if necessary, for all $y^{\pm}\in \overline{Y^{\pm}(r)}_A$ and all $w_{\pm}\in \overline{W^{\pm}(R)}_A$  one has
\[
|\Psi^{\pm}_{\sigma}(y^{\pm},w^{\pm})|\leq |\Psi^{\pm}_{\sigma}(y^{\pm},0_{W^{\pm}})| + q^{\pm}(R,L_1)|w^{\pm}|\leq a^{\pm}(R)r^{\pm}(R) + q^{\pm}(R,L_1)R<R.
\]
The hypotheses of Lemma \ref{lem:vbbanach} hold, yielding a unique $C^{k-1,1}$ section of $Y^{\pm}(r)_A\oplus \overline{W^{\pm}(R)}_A\rightarrow Y^{\pm}(r)_A$ which is fixed by $(y^{\pm},w^{\pm})\mapsto(y^{\pm},\Psi^{\pm}_{\sigma}(y^{\pm},w^{\pm}))$, hence giving a unique $C^{k-1,1}$ section $\sigma^{\pm}$ satisfying $\phi_{\sigma}(\tsig^{\pm}) = 0$.

\noindent{\textbf{Derivative estimates:}} Letting $\exp_E:E(R)_C\rightarrow M$ be the diffeomorphism onto its image defined by the exponential map for $M$, denote $\FC:=\exp_E^{-1}\circ f\circ \exp_E$ on $E(R)_A$. With respect to the identifications $\T_{(u_x,v_x)}E\cong \T_x X\oplus E^-_x\oplus E^+_x$ defined by the connections, $\T\FC$ identifies with the block operator
\[
\D\FC = \begin{pmatrix} \D^0_0\FC & \D^0_-\FC & \D^0_+\FC\\ \D^-_0\FC & \D^-_-\FC & \D^-_+\FC \\ \D^+_0\FC & \D^+_-\FC & \D^+_+\FC\end{pmatrix}
\]
with blocks
\begin{align}\label{eq:derivboundsFC}
\D^0_0\FC_{(0_x,0_x)} = \T_x g,\qquad \D^-_-\FC_{(0_x,0_x)} = \T^-_xf,\qquad \D^+_+\FC_{(0_x,0_x)} = \T^+_xf.
\end{align}
The equations $\phi_{\sigma}(\tsig^{\pm}) = 0$ then imply
\begin{align}\label{eq:jointmanifoldeq}
\FC\circ \sigma^- = \sigma\circ \FC^-\circ \sigma^-,\qquad \FC\circ\sigma = \sigma^+\circ\FC^+\circ\sigma.
\end{align}
These identities already imply $\sigma^{\pm}(0_x) = 0_x$. Covariant differentiation of the identities in \eqref{eq:jointmanifoldeq} yields
\begin{align}\label{eq:jointmanifoldeq2}
\AC_1^{\pm}[\nabla\sigma^{\pm}] = \BC_1^{\pm}[\nabla\sigma] + \CC_1^{\pm},
\end{align}
where
\[
\AC_1^{\pm}[T] = \begin{cases} (\D^+_+\FC|_{\sigma^-}-\nabla\sigma|_{\FC^-\circ\sigma^-}\D^{0-}_+\FC|_{\sigma^-})\,T&\text{ if $\pm=-$}\\T|_{\FC^+\sigma^+}\,(\D^{0+}_{0+}\FC|_{\sigma} + \D^{0+}_{-}\FC|_{\sigma}\nabla\sigma)&\text{ if $\pm=+$},\end{cases},
\]
\[
\BC_1^{\pm}[T] = \begin{cases} T|_{\FC^-\circ\sigma^-}\,\D^{0-}_{0-}\FC|_{\sigma^-}&\text{ if $\pm=-$}\\\D^-_-\FC|_{\sigma}\,T&\text{ if $\pm=+$}\end{cases},\quad\text{ and }\quad  \CC_1^{\pm} = \begin{cases} -\D^+_{0-}\FC|_{\sigma^-}&\text{ if $\pm=-$}\\\D^-_{0+}\FC|_{\sigma}&\text{ if $\pm=+$},\end{cases}
\]
from which it follows that $\nabla\sigma^{\pm}(0_x) = 0_x$ and 
\[
\|\nabla\sigma^{\pm}\|_{\infty}\leq \|(\AC_1^{\pm})^{-1}\|_{\infty}(\|\nabla\sigma\|_{\infty}\|\BC_1^{\pm}\|_{\infty} + \|\CC_1^{\pm}\|_{\infty})\leq (\rho^{\pm}_1 + \delta^{\pm}_1(R,L_1))L_1 + \epsilon^{\pm}_1(R)
\]
with $\delta^{\pm}_1(R,L_1),\,\epsilon^{\pm}_1(R) = o_R(1)$ by \eqref{eq:derivboundsFC}. Repeated covariant differentiation of \eqref{eq:jointmanifoldeq2} yields, for each $j=2,\dots,k$, polynomial functions
\[
\CC^{\pm}_j(\D\FC,\dots,\D^j\FC;\nabla\sigma,\dots,\nabla^{j-1}\sigma;\nabla\sigma^{\pm},\dots,\nabla^{j-1}\sigma^{\pm})
\]
and operators
\[
\AC^{\pm}_j[T]= \begin{cases} (D^+_+\FC|_{\sigma^-}-\nabla\sigma|_{\FC^-\circ\sigma^-}\D^{0-}_+\FC|_{\sigma^-})\,T&\text{ if $\pm=-$}\\T|_{\FC^+\circ\sigma}(\D^{0+}_{0+}\FC|_{\sigma}+\D^{0+}_-\FC|_{\sigma}\nabla\sigma)^{\otimes j}&\text{ if $\pm=+$}\end{cases}
\]
\[
\BC^{\pm}_j[T] = \begin{cases}T|_{\FC^-\circ\sigma^-}(\D^{0-}_{0-}\FC|_{\sigma^-} + \D^{0-}_{+}\FC|_{\sigma^-}\nabla\sigma^-)^{\otimes j}&\text{ if $\pm=-$}\\(\D^-_-\FC|_{\sigma} - \nabla\sigma^+|_{\FC^+\circ\sigma}\D^{0+}_-\FC|_{\sigma})\,T&\text{ if $\pm=+$}\end{cases}
\]
such that
\begin{align}\label{eq:geometrictransformderivatives}
\AC_j^{\pm}[\nabla^j\sigma^{\pm}] = \BC_j^{\pm}[\nabla^j\sigma] + \CC^{\pm}_j(\D\FC,\dots,\D^j\FC;\nabla\sigma,\dots,\nabla^{j-1}\sigma;\nabla\sigma^{\pm},\dots,\nabla^{j-1}\sigma^{\pm}),
\end{align}
from which
\[
\|\nabla^j\sigma^{\pm}\|_{\infty}\leq\|(\AC_j^{\pm})^{-1}\|_{\infty}(\|\nabla^j\sigma\|_{\infty}\|\BC_j^{\pm}\|_{\infty} + \|\CC^{\pm}_j\|_{\infty})\leq(\rho^{\pm}_j + \delta^{\pm}_j(R,L_1,L^{\pm}_1))L_j + \epsilon^{\pm}_j(R,\Lb_{1:j-1})
\]
follows with $\delta^{\pm}_j(R,L_1,L^{\pm}_1) = o_R(1)$ and $\epsilon^{\pm}_j(R,\mathbf{L}_{1:j-1}) = O_R(1)$ as claimed. Note that when $j=k$, these identities hold almost everywhere by Rademacher's theorem, which is sufficient to derive the Lipschitz bound.

\noindent{\textbf{Lipschitz estimate:}} We now demonstrate the claimed Lipschitz bound for the map $\sigma\mapsto\sigma^{\pm}$. Let $\sigma_1,\sigma_2\in S^{\pm}(A(R),R,\Lb)$, with associated fixed sections $\tau^{\pm}_1,\tau^{\pm}_2$ of $\Psi^{\pm}_{\sigma_1},\Psi^{\pm}_{\sigma_2}$ and projections $\tsig^{\pm}_1,\tsig^{\pm}_2$ to sections of $\MC\rightarrow Y^{\pm}$ respectively. Given $y^{\pm}\in Y^{\pm}(r^{\pm})_{A^{\pm}}$, let $w_i^{\pm}$ be the projection to $W^{\pm}$ of $\tsig^{\pm}_i(y^{\pm})$, so that $w^{\pm}_i= \Psi^{\pm}_{\sigma_i}(y^{\pm},w^{\pm}_i)$ for $i=1,2$. Then
\begin{align*}
|\tsig^{\pm}_1(y^{\pm})-\tsig^{\pm}_2(y^{\pm})|\leq |w^{\pm}_1-w^{\pm}_2|&\leq |\Psi^{\pm}_{\sigma_1}(y^{\pm},w^{\pm}_1)-\Psi^{\pm}_{\sigma_1}(y^{\pm},w^{\pm}_2)|+|\Psi^{\pm}_{\sigma_1}(y^{\pm},w^{\pm}_2)-\Psi^{\pm}_{\sigma_2}(y^{\pm},w^{\pm}_2)|\\&\leq q^{\pm}(R,L_1)|w^{\pm}_1-w^{\pm}_2|+ |\Psi^{\pm}_{\sigma_1}(y^{\pm},w^{\pm}_2)-\Psi^{\pm}_{\sigma_2}(y^{\pm},w^{\pm}_2)|.
\end{align*}
Letting $(\xi^{\pm}_{x^{\mp}},u^{\pm}_{x^{\mp}})$ denote the projection of $w^{\pm}_2$ to $g_{\mp}^*(\T X\oplus E^{\pm})$, it follows from the definition of $\Psi^{\pm}_{\sigma_2}$ that
\begin{align*}
(1-q^{\pm}(R,L_1))|w^{\pm}_1-w^{\pm}_2|&\leq |(L^{\pm})^{-1}(\Phi^{\pm}_{\sigma_1}(y^{\pm},w^{\pm}_2)-\Phi^{\pm}_{\sigma_2}(y^{\pm},w^{\pm}_2))|\\&\leq \begin{cases} |(L^-)^{-1}(\tsig_1(\xi^-_{x^+},u^-_{x^+})-\tsig_2(\xi^-_{x^+},u^-_{x^+}))|&\text{ if $\pm=-$}\\|(L^+)^{-1}(F\circ\tsig_1(\xi^+_{x^-},u^+_{x^-})-F\circ\tsig_2(\xi^+_{x^-},u^+_{x^-}))|&\text{ if $\pm=+$}\end{cases}\\&\leq \begin{cases} (\nu_2^{-1}+o_R(1))\|\sigma_1-\sigma_2\|_*|u^-_{x^+}|&\text{ if $\pm=-$}\\ (\nu_1+o_R(1))\|\sigma_1-\sigma_2\|_*|u^+_{x^-}|&\text{ if $\pm=+$}\end{cases}.
\end{align*}
Denoting $u^{\pm}_{x^{\pm}}$ denote the projection of $y^{\pm}$ to $E^{\pm}$, one has
\[
|u^{\pm}_{x^{\mp}}|\leq \begin{cases} (\nu_1 + o_R(1))|u^-_{x^-}|&\text{ if $\pm=-$}\\(\nu_2^{-1}+o_R(1))|u^+_{x^+}|&\text{ if $\pm=+$}.\end{cases}
\]
One thus has
\[
\frac{|\tsig^{\pm}_1(y^{\pm})-\tsig^{\pm}_2(y^{\pm})|}{|u^{\pm}_{x^{\pm}}|}\leq \frac{|w^{\pm}_1-w^{\pm}_2|}{|u^{\pm}_{x^{\pm}}|}\leq \kappa^{\pm}(R,L_1)\|\sigma_1-\sigma_2\|_*
\]
as claimed.
\end{proof}

Proposition \ref{prop:geometricgraphtransforms} assigns to each admissible section a unique geometric transform defined, in general, over only a \emph{strictly smaller} domain than the original section. To turn this geometric transform into a map on a space of sections with \emph{fixed domain}, we use a cutoff function, following \cite{bonatti_crovisier}.

Fix a relatively open, relatively pre-compact set $A^-\subset \mathrm{dom}(g)$, set $A^+:=g(A^-)$ and $A:=A^-\cup A^+$, and let $0<r\leq R$ be sufficiently small for Proposition \ref{prop:geometricgraphtransforms} to apply. Choose a smooth function $b^{\pm}\in C^{\infty}_c(A^{\pm}(r);[0,1])$ which is identically one on $A^{\pm}$ and identically zero outside $A^{\pm}(r/2)$. Also fix a positive function $r_1 = o_r(r)$, $h>0$ and, using Lemma \ref{lem:logarithmiccutoff}, $0<r_0<r_1$ and a logarithmic cutoff function $\varphi_{h}:[0,\infty)\rightarrow [0,1]$ 
\begin{align}\label{eq:logarithmiccutoffproperties}
\varphi_{h}|_{[0,r_0]}\equiv 1,\qquad \varphi_{h}|_{[r_1,\infty)}\equiv 0,\qquad \sup_{s\in[0,\infty)}s|\nabla\varphi_{h}(s)|\leq h.
\end{align}
Then define a cutoff function $\chi_h^{\pm}:E^{\pm}(R)_{A(R)}\rightarrow[0,1]$ defined by extending
\[
\chi_h^{\pm}(w_x):=b^{\pm}(x)\varphi_{h}(|w_x|)
\]
by zero. Given $\Lb\in\RB^k_{\geq0}$, define the \emph{analytic graph transform} $G_h^{\pm}(\sigma)$ of any $\sigma\in S^{\pm}(A(R),R,\Lb)$ by
\[
G_h^{\pm}(\sigma)(w_x):=\begin{cases}\chi_h^{\pm}(w_x)\sigma^{\pm}(w_x)&\text{ $w_x\in E^{\pm}(r)_{A^{\pm}(r)}$}\\0&\text{ otherwise},\end{cases}
\]
where $\sigma^{\pm}$ is the geometric graph transform of $\sigma$ given by Proposition \ref{prop:geometricgraphtransforms}. The following result is inspired by the elegant argument of \cite{center_regularity}.

\begin{proposition}\label{prop:analyticgraphtransform}
Assume that $\rho^{\pm}_j<1$ for all $j=1,\dots,k$. Then for all sufficiently small $R>0$, there is $h>0$ and $\Lb^{\pm}\in\RB^k_{\geq 0}$ such that $G^{\pm}_h:S^{\pm}(A(R),R,\Lb^{\pm})\rightarrow S^{\pm}(A(R),R,\Lb^{\pm})$ is a contraction in the $\|\cdot\|_*$ norm. Consequently, $G^{\pm}_h$ admits a unique fixed section $\sigma^{\pm}_*\in S^{\pm}(A(R),R,\Lb^{\pm})$, and this section is $C^k$.
\end{proposition}

\begin{proof}
Choose $L_1$ sufficiently small that $\max\{\nu_1\nu_2^{-1}(1-\nu_2^{-1}L_1)^{-1},\nu_1\nu_2^{-1}(1-\nu_1L_1)^{-1}, L_1\}<1$. Then, shrinking $R$ if necessary, one has $\kappa^{\pm}(R,L_1)<1$. Hence for any $\sigma_1,\sigma_2\in S^{\pm}(A(R),R,L_1)$ one has
\[
\|G^{\pm}_h(\sigma_1)-G^{\pm}_h(\sigma_2)\|_*\leq \|\chi^{\pm}_h(\sigma_1^{\pm}-\sigma_2^{\pm})\|_*\leq \kappa^{\pm}(R,L_1)\|\sigma_1-\sigma_2\|_*
\]
by Proposition \ref{prop:geometricgraphtransforms}, making $G^{\pm}_h$ a contraction.

To see that $G^{\pm}_h$ preserves an appropriately chosen $S^{\pm}(A(R),R,L_1)$, fix $\sigma\in S^{\pm}(A(R),R,L_1)$ and compute
\[
\nabla(G^{\pm}_h(\sigma)) = \chi^{\pm}_h\nabla\sigma^{\pm} + \D\chi^{\pm}_h\,\sigma^{\pm} = \chi^{\pm}_h\nabla\sigma^{\pm} + \nabla b^{\pm}\,\varphi_{h}\,\sigma^{\pm} + b^{\pm}\,\nabla\varphi_{h}\,\sigma^{\pm}.
\]
Since $\sigma^{\pm}(0_x) = 0_x$ and $\nabla\sigma^{\pm}(0_x) = 0_x$ wherever defined and since $G^{\pm}_h(\sigma)$ vanishes outside of this domain, it follows from this equation $G^{\pm}_h(\sigma)$ vanishes to first order at every point of the zero section. Moreover, for $w_x\in\mathrm{supp}(\chi^{\pm})$, observe that
\[
|\nabla b^{\pm}(x)\,\varphi_h(|w_x|)\,\sigma^{\pm}(w_x)|\leq |\nabla b^{\pm}(x)|\,r_1\,\|\sigma^{\pm}\|_* = O_r(r^{-1})\,o_r(r)\,L_1 = o_r(1) = o_R(1),
\]
while
\[
|b^{\pm}(x)\,\nabla\varphi_h(|w_x|)\,\sigma^{\pm}(w_x)|\leq |w_x|\,|\D\varphi_h(|w_x|)|\,\|\sigma^{\pm}\|_*\leq h L_1
\]
by \eqref{eq:logarithmiccutoffproperties}. Consequently, by Proposition \ref{prop:geometricgraphtransforms},
\[
\|\nabla(G^{\pm}_h(\sigma))\|_{\infty}\leq (\rho^{\pm}_1 + \delta^{\pm}_1(R,L_1) + h)L_1 + \epsilon^{\pm}_1(R,L_1) \leq L_1
\]
with $\delta^{\pm}_1(R,L_1),\epsilon^{\pm}_1(R,L_1) =o_R(1)$. Shrinking $h$ and $R$ if necessary one sees that $G^{\pm}_h$ preserves $S^{\pm}(A(R),R,L_1)$.

We now consider the action of $G^{\pm}_h$ on the smoother space $S^{\pm}(A(R),R,\Lb)$, assuming $k\geq 2$. Shrink $R$ yet further if necessary so that the coefficients $\rho^{\pm}_j + \delta^{\pm}_j(R,L_1)$ of Proposition \ref{prop:geometricgraphtransforms} are all strictly less than one. For $2\leq j\leq k$, suppose that $L_1,\dots,L_{j-1}$ have already been chosen such that $G^{\pm}_h$ preserves $S^{\pm}(A(R),R,\Lb_{1:j-1})$. One has
\begin{align}\label{eq:nablajG}
\nabla^j(G^{\pm}_h(\sigma)) = \chi^{\pm}_h\,\nabla^j\sigma^{\pm} + \sum_{i=1}^j\sum_{s\in\mathrm{Sh}(j,j-i)}s\big(\nabla^i\chi^{\pm}_h\otimes\nabla^{j-i}\sigma^{\pm}),
\end{align}
where $\mathrm{Sh}(j,j-i)$ denotes the set of $(j,j-i)$-shuffles. It follows from Proposition \ref{prop:geometricgraphtransforms} that
\[
\|\nabla^j(G^{\pm}_h(\sigma))\|_{\infty}\leq (\rho^{\pm}_j + \delta^{\pm}_j(R,L_1))L_j + Q^{\pm}_j(R,\Lb_{1:j-1})
\]
for some non-negative function $Q^{\pm}_j(R,\Lb_{1:j-1})$. Choosing
\[
L_j\geq \frac{Q^{\pm}_j(R,\Lb_{1:j-1})}{1-\rho^{\pm}_j-\delta^{\pm}_j(R,L_1)}
\]
then gives $\|\nabla^j(G^{\pm}_h(\sigma))\|_{\infty}\leq L_j$ as desired.

Finally we prove that the unique fixed point $\sigma^{\pm}_*$ of $G^{\pm}_h$ is $C^k$. Let $\sigma^{\pm}_n$ be the $n^{th}$ iterate of the zero section under $G^{\pm}$; since $F$ is $C^k$, $\sigma^{\pm}_n$ is $C^k$ for all $n\in\NB$. The sequence $\{\nabla^k\sigma^{\pm}_n\}_{n=1}^{\infty}$ is uniformly bounded by $L_k$. It is also uniformly equicontinuous. Indeed, for each $n\in\NB$ define the modulus
\[
\omega_n(t):=\sup_{d(u,v)\leq t}|\nabla^k\sigma^{\pm}_n(u)-P_{v\rightarrow u}\nabla^k\sigma^{\pm}_n(v)|,\qquad\text{ $t$ sufficiently small}.
\]
Having fixed $R$ as above, denote $a^{\pm}_k:=\rho^{\pm}_k + \delta^{\pm}_k(R,L_1)$ as in Proposition \ref{prop:geometricgraphtransforms}. Invoking \eqref{eq:nablajG} and the fact that $\|\chi^{\pm}\|_{\infty}=1$, commuting parallel transport with the operator $(\AC^{\pm}_k)^{-1}\BC^{\pm}_k$ of the identity \eqref{eq:geometrictransformderivatives} at the cost of an $o_t(1)$ term and noting that $(\AC^{\pm}_k)^{-1}\BC^{\pm}_k$ involves pre-composition with a reparametrisation ($\FC^-\sigma^-$ when $\pm=-$, $(\FC^+\sigma)^{-1}$ when $\pm=+$), one obtains a constant $K_{\pm}>0$ and non-negative function $\beta_n(t) = o_t(1)$ uniformly in $n$ such that
\begin{align}\label{eq:omegasigma}
\omega_{n+1}(t)\leq a^{\pm}_k\omega_{n}(K_{\pm}t) + \beta_n(t)
\end{align}
for all sufficiently small $t\geq 0$. To prove uniform equicontinuity of $\{\nabla^k\sigma^{\pm}_n\}_{n=1}^{\infty}$, it suffices to prove that $\sup_n\omega^{\pm}_n(t) = o_t(1)$. Fix $\epsilon>0$. Iterating \eqref{eq:omegasigma} $N$ times gives
\begin{align*}
\omega^{\pm}_n(t)\leq (a^{\pm}_k)^N\omega^{\pm}_{n-N}(K_{\pm}^Nt) + \sum_{j=0}^{N-1}(a^{\pm}_k)^j\beta(K_{\pm}^jt)\leq 2L_k(a^{\pm}_k)^N + \sum_{j=0}^{N-1}(a^{\pm}_k)^j\beta(K_{\pm}^jt)
\end{align*}
for any $n\geq N$. Since $a^{\pm}_k<1$, $N$ may be chosen sufficiently large that $2L_k(a^{\pm}_k)^N<\epsilon/2$. Having fixed such $N$, choose $t_0>0$ sufficiently small that $\sup_{t\in[0,t_0]}\sum_{j=0}^{N-1}(a^{\pm}_k)^j\beta(K_{\pm}^jt)<\epsilon/2$ and $\sup_{t\in[0,t_0]}\max_{j=1,\dots,N-1}\omega_j(t)<\epsilon$. Then $\sup_{t\in[0,t_0]}\sup_n\omega_n(t)<\epsilon$. Since $\{\nabla^k\sigma^{\pm}_n\}_{n=1}^{\infty}$ is thus uniformly bounded and uniformly equicontinuous, Lemma \ref{lem:vbaa} applies to give $\sigma^{\pm}_*$ being $C^k$. 
\end{proof}

Theorem \ref{thm:invariantmanifolds} now follows easily.

\begin{theorem}\label{thm:invariantmanifoldsmain}
If $\rho^{\pm}_j<1$ for all $1\leq j\leq k$, then for any relatively open, relatively pre-compact set $A^-\subset g^{-1}(X)\cap X$, there is a locally $f$-invariant $C^k$ manifold $N^{\pm}$ containing $A^{\pm}$ and tangent at $A^{\pm}$ to $\T A^{\pm}+E^{\pm}$.
\end{theorem}

\begin{proof}
Proposition \ref{prop:analyticgraphtransform} gives a $C^{k}$ section $\sigma^{\pm}_*$ of $E(r)_{A^{\pm}}\rightarrow E^{\pm}(r)_{A^{\pm}}$ for some $r>0$ satisfying $\nabla\sigma^{\pm}_*(0_{A^{\pm}}) = 0_{A^{\pm}}$, and whose lift $\tsig^{\pm}_*$ satisfies
\[
F\circ\tsig^{\pm}_* = \tsig^{\pm}_*\circ F^{\T,\pm}\circ \tsig^{\pm}_*
\]
on a small tube in $Y^{\pm}_{A^{\pm}\cap g^{-1}(A^{\pm})}$. Then $N^{\pm}:=\exp_E(\sigma^{\pm}_*(E^{\pm}(r)_{A^{\pm}}))$ is a $C^k$ manifold containing $A^{\pm}$ and tangent at $A^{\pm}$ to $\T A^{\pm}+E^{\pm}$ for which $f(N^{\pm})\cap \exp_E(E(r)_{A^{\pm}})\subset N^{\pm}$.
\end{proof}

\section{Locally invariant foliation theorems}\label{sec:invariantfoliations}

In this section, we assume that $k\geq1$ and that $f$ and $E^{\pm}\rightarrow X$ are all $C^{k,1}$. Given relatively open, relatively pre-compact $C\subset X$, observe that any foliation with leaves parametrised by points in $C$ and tangent at $C$ to $E^{\pm}_C$ may be realised as the image under $\Theta$ of the graph of a section of $\MC(R)_C\rightarrow E^{\pm}(R)_C$ which vanishes to first order on the zero section. Our invariant foliations will be constructed by performing a graph transform argument on an appropriate space of such sections.

Unlike in the manifold case, we must here distinguish between different directions in the derivatives of sections. For any $1\leq j\leq k$, one has an isomorphism
\begin{align}\label{eq:mixedderivatives}
\big(\T X\oplus E^{\pm}\big)^{\otimes j} = \T X^{\otimes j}\oplus \bigoplus_{b=1}^j\bigoplus_{s\in\mathrm{Sh}(j-b,b)}s\bigg(\T X^{\otimes (j-b)}\otimes (E^{\pm})^{\otimes b}\bigg),
\end{align}
where $\mathrm{Sh}(j-b,b)$ is the set of $(j,b)$-shuffles. Given a $C^k$ section $\sigma$ of $\MC(R)_C\rightarrow E^{\pm}(R)_C$, we write $\nabla_0^j\sigma$ for the component of $\nabla^j\sigma$ acting on $\T X^{\otimes j}$ and $\nabla^j_{\setminus 0}\sigma$ for the component of $\nabla^j\sigma$ acting on the other summands. Given $\Lb_0,\Lb_{\setminus0}\in\RB^k_{\geq 0}$, we then denote by $\SC^{\pm}(C,R,\Lb_0,\Lb_{\setminus0})$ the space of sections $\sigma$ of $C^{k-1,1}$ sections of $\MC(R)_C\rightarrow E^{\pm}(R)_C$ with are differentiable at every point of the zero section over $C$, satisfying
\[
\sigma(0_x) = 0_x,\qquad \nabla\sigma(0_x) = 0_x,\qquad \forall x\in C,
\]
with derivative bounds
\[
\|\nabla^j_0\sigma\|_{*}:=\esssup_{u_x\in E^{\pm}(R)_C\setminus 0_C}\frac{|\nabla^j_0\sigma(u_x)|}{|u_x|}\leq L_{j,0},\qquad \|\nabla^j_{\setminus 0}\sigma\|_{\infty}\leq L_{j,\setminus0},\qquad \forall j=1,\dots,k
\]
and with norm
\[
\|\sigma\|_*:=\sup_{u_x\in E^{\pm}(R)_C\setminus 0_C}\frac{|\sigma(u_x)|}{|u_x|}
\]
being finite. The proof of the following follows from the same argument as that of Proposition \ref{prop:manifoldsectionscomplete}.

\begin{proposition}\label{prop:foliationsectionscomplete}
For any relatively open, relatively pre-compact $C\subset X$ and any $R>0$, $\Lb_0,\Lb_{\setminus0}\in\RB^k_{\geq0}$, the space $(\SC^{\pm}(C,R,\Lb_0,\Lb_{\setminus0}),\|\cdot\|_*)$ is a complete metric space.
\end{proposition}

As when discussing invariant manifolds, we will refer to elements of the section space of Proposition \ref{prop:foliationsectionscomplete} as \emph{admissible sections}. Note that for any admissible section $\sigma$, the condition $\|\nabla_0\sigma\|_*<\infty$ is sufficient to guarantee that the image of the graph of $\sigma$ under $\Theta$ is indeed a foliation near $C$, since it implies that the local Lipschitz constant of $\sigma$ in the base direction is arbitrarily small on any sufficiently small tube. For narrational convenience, given $x\in C$ we will refer to the image of the fibre $E^{\pm}(R)_x$ under $\sigma$ as the \emph{leaf} of $\sigma$ through $x$.

We now discuss how the geometric graph transform is to be defined in this foliation setting. Given an admissible section $\sigma$ defining some local foliation, the goal is to obtain some other local foliation whose leaves are either \emph{mapped into} those of $\sigma$ by $F$ (in the stable case) or \emph{mapped to} by those of $\sigma$ under $F$ (in the unstable case). To be precise, let $A^-\subset g^{-1}(X)\cap X$ be relatively open and relatively pre-compact, $A^+:=g(A^-)$ and $C:=A^-\cup A^+$, and let $R>0$ be small. Given an admissible section $\sigma$ of $\MC(R)_C\rightarrow E^{\pm}(R)_C$, we seek $0<r^{\pm}\leq R$ and a section $\sigma^{\pm}$ of $\MC(r^{\pm})_C\rightarrow E^{\pm}(r^{\pm})_C$ such that (i) when $\pm=-$, the leaves of $\sigma$ contain those of $F\circ\sigma^-$; (ii) when $\pm=+$, the leaves of $F\circ\sigma$ contain those of $\sigma^+$. Denoting points in $A^-$ by $x^-$ and their images under $g$ in $A^+$ by $x^+$, in the former case it suffices to assign to any $u_{x^-}\in E^-(r^-)_{x^-}$ a unique pair $(\xi'_{x^-},v'_{x^-})\in(\T X\oplus E^+)_{x^-}$ and unique $u'_{x^+}\in E^-(R)_{x^+}$ such that
\[
F(\xi'_{x^-},u_{x^-},v'_{x^-}) = \sigma(u'_{x^+});
\]
the pair $(\xi'_{x^-},v'_{x^-})$ would then be the fibre value taken by $\sigma^-$ at $u_{x^-}$ and $u'_{x^+}$ would be the base value obtained by evaluating $F$ on $(\xi'_{x^-},u_{x^-},v'_{x^-})$. In the latter case, it suffices to assign to any $v_{x^+}\in E^+(r^+)_{x^+}$ a unique pair $(\xi'_{x^+},u'_{x^+})\in (\T X\oplus E^-)_{x^+}$ and unique $v'_{x^-}\in E^+(R)_{x^-}$ such that
\[
F\circ\sigma(v'_{x^-}) = (\xi'_{x^+},u'_{x^+},v_{x^+});
\]
the pair $(\xi'_{x^+},u'_{x^+})\in (\T X\oplus E^-)_{x^+}$ would then be the fibre value taken by $\sigma^+$ at $v_{x^+}$ while $v'_{x^-}$ would be the base value whose image under $\sigma$ maps by $F$ to $(\xi'_{x^+},u'_{x^+},v_{x^+})$.

We therefore define
\[
\WC^{\pm}:=\begin{cases} \T X\oplus g_+^*E^-\oplus E^+&\text{ if $\pm=-$}\\\T X\oplus E^-\oplus g_-^*E^+&\text{ if $\pm=+$}\end{cases},\qquad \ZC^{\pm}:=E^{\pm}\oplus\WC^{\pm},
\]
and, given admissible $\sigma$, define $\Phi^{\pm}_{\sigma}:\ZC(R)_{A^{\pm}}\rightarrow \MC_{A^{\pm}}$ by
\begin{align}\label{eq:Phiminusfoliation}
\Phi_{\sigma}^{-}(u_{x^-},\xi'_{x^-},u'_{x^+},v'_{x^-}):=F(\xi'_{x^-},u_{x^-},v'_{x^-})-\sigma(u'_{x^+}),\qquad \forall (u_{x^-},\xi'_{x^-},u'_{x^+},v'_{x^-})\in\ZC^{-}(R)_{A^{-}}
\end{align}
and
\begin{align}\label{eq:Phiplusfoliation}
\Phi^+_{\sigma}(v_{x^+},\xi'_{x^+},u'_{x^+},v'_{x^-}):=F\circ\sigma(v'_{x^-})-(\xi'_{x^+},u'_{x^+},v_{x^+}),\qquad \forall (v_{x^+},\xi'_{x^+},u'_{x^+},v'_{x^-})\in\ZC^+(R)_{A^+}.
\end{align}
It suffices to show that the zero set of $\Phi^{\pm}_{\sigma}$, restricted to a suitable domain, is the graph of a unique section $\tau^{\pm}$, and that this section is admissible. Denoting $\sigma^{\pm}$ for the projection of the $\tau^{\pm}$ onto $\MC$ and
\[
\psi^{\pm}_{\sigma}(\eta):=\begin{cases} F\circ\eta-\sigma\circ F^-\circ\eta&\text{ if $\pm=-$}\\F\circ\sigma-\eta\circ F^+\circ\sigma&\text{ if $\pm=+$}\end{cases},
\]
this $\sigma^{\pm}$ would then satisfy the familiar graph transform identity $\psi^{\pm}(\sigma^{\pm})\equiv 0$. Existence, uniqueness and character of this $\sigma^{\pm}$ is the subject of the next proposition.

\begin{proposition}[Foliation geometric graph transforms]\label{prop:geometricgraphtransformsfoliation}
Let $A^-\subset f^{-1}(X)\cap X=\operatorname{dom}(g)$ be relatively open, relatively pre-compact, define $A^+:=g(A^-)$ and $A:=A^-\cup A^+$, and for each $j=1,\dots,k$, define
\[
\gamma^{\pm}_{j,0}:=\begin{cases} \nu_1\min\{\mu_1,\nu_2\}^{-1}\mu_2^j&\text{ if $\pm=-$}\\\nu_2^{-1}\max\{\mu_2,\nu_1\}\mu_1^{-j}&\text{ if $\pm=+$}\end{cases},\qquad \gamma^{\pm}_{j,\setminus0}:=\begin{cases} \nu_1\min\{\mu_1,\nu_2\}^{-1}\max\{\mu_2,\nu_1\}^{j-1}&\text{ if $\pm=-$}\\\nu_2^{-1}\max\{\mu_2,\nu_1\}\min\{\mu_1,\nu_2\}^{-(j-1)}&\text{ if $\pm=+$}\end{cases}.
\]
Let $\mathbf{L}_{0} = (L_{1,0},\dots,L_{k,0})\in \RB^{k}_{\geq 0}$ and $\Lb_{\setminus 0} = (L_{1,\setminus0},\dots,L_{k,\setminus0})\in\RB^k_{\geq 0}$ satisfy $\max\{\mu_1^{-1},\nu_2^{-1},\nu_1,\mu_2\}L_{1,\setminus0}<1$. Then there are positive functions $r^{\pm}(R) = \Theta(R)\leq R$, functions
\begin{align*}
L^{\pm}_{1,\setminus0}&\equiv L^{\pm}_{1,\setminus0}(R,L_{1,\setminus0}) = (\gamma^{\pm}_{1,\setminus0} + \delta^{\pm}_{1,\setminus0}(R,L_{1,\setminus0}))L_{1,\setminus0} + \epsilon^{\pm}_{1,\setminus0}(R,L_{1,\setminus0}) \\ L^{\pm}_{1,0}&\equiv L^{\pm}_{1,0}(R,L_{1,\setminus0},L_{1,0})  = (\gamma^{\pm}_{1,0} + \delta^{\pm}_{1,0}(R,L_{1,\setminus0},L^{\pm}_{1,\setminus0}))L_{1,0} + \epsilon^{\pm}_{1,0}(R,L_{1,\setminus0},L^{\pm}_{1,\setminus0})
\end{align*}
with $\delta^{\pm}_{1,0},\delta^{\pm}_{1,\setminus0},\epsilon^{\pm}_{1,\setminus0} = o_R(1)$ and $\epsilon^{\pm}_{1,0} = O_R(1)$; and, if $k\geq 2$, for each $2\leq j\leq k$ functions
\begin{align*}
L^{\pm}_{j,\setminus0}&\equiv L^{\pm}_{j,\setminus0}(R,\Lb_{1:j-1,0},\Lb_{1:j-1,\setminus0},\Lb_{1:j-1,0}^{\pm},\Lb_{1:j-1,\setminus0}^{\pm})\\ &= (\gamma^{\pm}_{j,\setminus0} + \delta^{\pm}_{j,\setminus0}(R,L_{1,0},L_{1,\setminus0})L_{j,\setminus0} + \epsilon^{\pm}_{j,\setminus0}(R,\Lb_{1:j-1,0},\Lb_{1:j-1,\setminus0},\Lb_{1:j-1,0}^{\pm},\Lb_{1:j-1,\setminus0}^{\pm})\\L^{\pm}_{j,0}&\equiv L^{\pm}_{j,0}(R,\Lb_{1:j-1,0},\Lb_{1:j,\setminus0},\Lb_{1:j-1,0}^{\pm},\Lb_{1:j,\setminus0}^{\pm})\\ &= (\gamma^{\pm}_{j,0} + \delta^{\pm}_{j,0}(R,L_{1,0},L_{1,\setminus0}))L_{j,0} + \epsilon^{\pm}_{j,0}(R,\Lb_{1:j-1,0},\Lb_{1:j,\setminus0},\Lb_{1:j-1,0}^{\pm},\Lb_{1:j,\setminus0}^{\pm})
\end{align*}
with $\delta^{\pm}_{j,0},\delta^{\pm}_{j,\setminus0} = o_R(1)$ and $\epsilon^{\pm}_{j,\setminus0},\epsilon^{\pm}_{j,0} = O_R(1)$ such that every $\sigma\in \SC^{\pm}(A,R,\Lb_0,\Lb_{\setminus 0})$ admits a unique $\sigma^{\pm}\in \SC^{\pm}(A^{\pm},r^{\pm},\Lb^{\pm}_0,\Lb^{\pm}_{\setminus0})$ such that $\psi^{\pm}_{\sigma}(\sigma^{\pm})\equiv 0$. Furthermore, there is
\[
\kappa^{\pm}(R,L_{1,\setminus0}):=\begin{cases} \nu_1\min\{\mu_1,\nu_2\}^{-1}(1-\min\{\mu_1,\nu_2\}^{-1}L_{1,\setminus0})^{-1} + o_R(1)&\text{ if $\pm=-$}\\ \nu_2^{-1}\max\{\mu_2,\nu_1\}(1-\max\{\mu_2,\nu_1\}L_{1,\setminus0})^{-1} +o_R(1)&\text{ if $\pm=+$}\end{cases},
\]
such that
\[
\|\sigma_1^{\pm}-\sigma_2^{\pm}\|_*\leq \kappa^{\pm}(R,L_{1,\setminus0})\|\sigma_1-\sigma_2\|_*
\]
for all $\sigma_1,\sigma_2\in \SC^{\pm}(A,R,\Lb_0,\Lb_{\setminus0})$.
\end{proposition}

\begin{proof}
The proof follows the same structure as that of Proposition \ref{prop:geometricgraphtransforms}: first we prove existence and uniqueness of $\sigma^{\pm}$ given $\sigma$, then we prove the derivative bounds, and finally we prove the Lipschitz-in-$\|\cdot\|_*$ estimate.

\vspace{2mm}

\noindent\textbf{Existence and uniqueness of $\sigma^{\pm}$:} Fixing small $R>0$ and $\sigma\in\SC^{\pm}(A,R,\Lb_,\Lb_{\setminus0})$, let $\Phi^{\pm}_{\sigma}:\overline{\ZC(R)}_{A^{\pm}}\rightarrow \MC_{A^{\pm}}$ be defined as in \eqref{eq:Phiminusfoliation}, \eqref{eq:Phiplusfoliation}. Denote the derivative of $\Phi^{\pm}_{\sigma}$ with respect to the final three variables, expressed with respect to the identitications induced by the connection, by $\V\Phi^{\pm}_{\sigma}$. Since $\sigma$ is at least Lipschitz, this derivative is defined almost everywhere by Rademacher's theorem and is given by
\begin{align}
\V\Phi^-_{\sigma} = \begin{pmatrix} \D^{\T}_{\T} F & -\nabla^{\T}_-\sigma & \D^{\T}_+F \\ \D^-_{\T}F & -I_{E^-} & \D^-_+F \\ \D^+_{\T}F & -\nabla^+_-\sigma & \D^+_+F\end{pmatrix},\qquad \V\Phi^+_{\sigma} = \begin{pmatrix} -I_{\T X} & 0 & \D^{\T}_+F + \D^{\T}_{\T,-}F\,\nabla_+\sigma \\ 0 & -I_{E^-} & \D^-_+F + \D^-_{\T,-}F\,\nabla_+\sigma \\ 0 & 0 & \D^+_+F + \D^+_{\T,-}F\,\nabla_+\sigma\end{pmatrix}
\end{align}
wherever defined. Using compactness of $\overline{A}$, choose a $C^k$ vector bundle morphism $L^{\pm}:\WC^{\pm}\rightarrow g_+^*\MC$ which is uniformly within $R$ of $\T g\oplus (-I_{E^-})\oplus \T^+f$ (if $\pm=-$) or $(-I_{\T X})\oplus (-I_{E^-})\oplus \T^+f$ (if $\pm=+$) over $A$ in the fibrewise operator norm. Then define the fibre bundle morphism $\Psi^{\pm}_{\sigma}:\overline{\ZC^{\pm}(R)}_{A^{\pm}}\rightarrow \WC^{\pm}_{A^{\pm}}$ by the formula
\[
\Psi^{\pm}_{\sigma}(y^{\pm},w^{\pm}):=w^{\pm}-(L^{\pm})^{-1}\Phi^{\pm}_{\sigma}(y^{\pm},w^{\pm})
\]
for $y^{\pm}\in E^{\pm}(R)_{A^{\pm}}$ and $w^{\pm}\in \ZC^{\pm}(R)_{y^{\pm}}$. The zero-set of $\Phi^{\pm}_{\sigma}$ is precisely the fixed set of $\Psi^{\pm}_{\sigma}$, which we now demonstrate to be given by the graph of a unique $C^{k-1,1}$ section using Lemma \ref{lem:vbbanach}.

To apply the lemma, we must check that $\Psi^{\pm}_{\sigma}$ is uniformly vertically contractive and preserves a fibrwise tube. For the former, observe that at any point where $\Psi^{\pm}_{\sigma}$ is differentiable, one has
\[
|\V\Psi^{\pm}_{\sigma}| =  |(L^{\pm})^{-1}(L^{\pm}-\V\Phi^{\pm}_{\sigma})|\leq q^{\pm}(R,L_{1,\setminus0}):=\begin{cases} \min\{\mu_1,\nu_2\}^{-1}L_{1,\setminus0} + o_R(1)&\text{ if $\pm=-$}\\\max\{\mu_2,\nu_1\}L_{1,\setminus0} + o_R(1)&\text{ if $\pm=+$}\end{cases}
\]
uniformly over $\ZC^{\pm}(R)_{A^{\pm}}$ and $\sigma\in\SC^{\pm}(A,R,\Lb_0,\Lb_{\setminus0})$. Shrinking $R$ if necessary and invoking the hypothesis on $L_{1,\setminus0}$ gives $q^{\pm}(R,L_{1,\setminus0})<1$, making $\Psi^{\pm}_{\sigma}$ a fibrewise contraction. For the latter, observe that
\[
\Psi^-_{\sigma}(u_{x^-},0_{\WC^{-}}) = |(L_{x^-}^-)^{-1}F(0_{x^-},u_{x^-},0_{x^-})|\leq a^-(R)|u_{x^-}|
\]
and
\[
\Psi^+_{\sigma}(v_{x^+},0_{\WC^{+}}) = |(L_{x^+}^+)^{-1}(0_{x^+},0_{x^+},-v_{x^+})|\leq a^+(R)|v_{x^+}|
\]
where
\[
a^{\pm}(R):=\begin{cases} \nu_1 + o_R(1)&\text{ if $\pm=-$}\\\nu_2^{-1} + o_R(1)&\text{ if $\pm=+$}\end{cases}.
\]
Fixing $\alpha\in (0,1)$ and defining
\[
r^{\pm}(R,L_{1,\setminus0}):=\min\bigg\{R,\alpha\frac{1-q^{\pm}(R,L_{1,\setminus0})}{a^{\pm}(R)}R\bigg\},
\]
shrinking $R$ yet further if necessary, for all $y^{\pm}\in \overline{E^{\pm}(r^{\pm})}_{A^{\pm}}$ and all $w^{\pm}\in\overline{\ZC^{\pm}(R)}_{y^{\pm}}$, one has
\[
|\Psi^{\pm}_{\sigma}(y^{\pm},w^{\pm})|\leq |\Psi^{\pm}_{\sigma}(y^{\pm},0_{\WC^{\pm}})| + q^{\pm}(R,L_{1,\setminus0})|w^{\pm}|\leq a^{\pm}(R)r^{\pm}(R,L_{1,\setminus0}) + q^{\pm}(R,L_{1,\setminus0})R<R.
\]
Thus the hypotheses of Lemma \ref{lem:vbbanach} hold, yielding a unique section $\sigma^{\pm}$ of $E^{\pm}(r)_{A^{\pm}}\oplus\overline{\WC^{\pm}(R)}_{A^{\pm}}\rightarrow E^{\pm}(r)_{A^{\pm}}$ which is fixed by $(y^{\pm},w^{\pm})\mapsto (y^{\pm},\Psi^{\pm}_{\sigma}(y^{\pm},w^{\pm}))$, and this section is $C^{k-1,1}$.

\vspace{2mm}

\noindent\textbf{Derivative estimates: } Since the graph of $\sigma^{\pm}$ is annihilated by $\Phi^{\pm}_{\sigma}$, one has the familiar graph transform identities
\begin{align}\label{eq:foliationidentities}
F\circ\sigma^- = \sigma\circ F^{-}\circ \sigma^-,\qquad F\circ\sigma = \sigma^+\circ F^+\circ\sigma
\end{align}
on the appropriate domains. Covariant differentiation of \eqref{eq:foliationidentities} yields
\begin{align*}
\AC^{\pm}_1[\nabla\sigma^{\pm}] = \BC^{\pm}_1[\nabla\sigma] + \CC^{\pm}_1,
\end{align*}
where
\begin{align*}
\AC^{\pm}_1[T] = \begin{cases} (\D^{\T+}_{\T+}F|_{\sigma^-}-\nabla\sigma|_{F^{0-}\sigma^-}\,\D^{0-}_{\T+}F|_{\sigma^-})\,T&\text{ if $\pm=-$}\\ T|_{F^{0+}\sigma}(\D^{0+}_{0+}F|_{\sigma} + \D^{0+}_{\T-}F|_{\sigma}\,\nabla\sigma)&\text{ if $\pm=+$}\end{cases},
\end{align*}
and
\begin{align*}
\BC^{\pm}_1[T]= \begin{cases} T|_{F^{0-}\sigma^-}\,\D^{0-}_{0-}F|_{\sigma^-}&\text{ if $\pm=-$}\\\D^{\T-}_{\T-}F|_{\sigma}\,T&\text{ if $\pm=+$}\end{cases},\qquad \CC^{\pm}_1 = \begin{cases} -\D^{\T+}_{0-}F|_{\sigma^-}&\text{ if $\pm=-$}\\\D^{\T-}_{0+}F|_{\sigma}&\text{ if $\pm=+$}\end{cases}.
\end{align*}
Letting $\VC^j_{\pm}$ denote the summand of \eqref{eq:mixedderivatives} complementary to $\T X^{\otimes j}$, therefore, one has
\begin{align*}
\|\nabla_{\setminus 0}\sigma^{\pm}\|_{\infty}&\leq \|(\AC^{\pm}_1)^{-1}\BC^{\pm}_1[\nabla\sigma]|_{\VC^1_{\pm}}\|_{\infty} + \|(\AC^{\pm}_1)^{-1}\CC^{\pm}_1|_{\VC^1_{\pm}}\|_{\infty}\\&\leq(\gamma^{\pm}_{1,\setminus0} + \delta^{\pm}_{1,\setminus0}(R,L_{1,\setminus0}))L_{1,\setminus0} + \epsilon^{\pm}_{1,\setminus0}(R,L_{1,\setminus0})
\end{align*}
with $\delta^{\pm}_{1,\setminus0}(R,L_{1,\setminus0}) = o_R(1)$ and $\epsilon^{\pm}_{1,\setminus0}(R,L_{1,\setminus0}) = o_R(1)$. Restricting instead to $\T X$ and using the fact that $F$ is a bundle map, so that $\D^0_iF=0$ for any $i\in\{\T,-,+\}$, yields
\begin{align*}
\AC^{\pm}_{1,0}[\nabla_0\sigma^{\pm}] = \BC^{\pm}_{1,0}[\nabla_0\sigma] + \CC^{\pm}_{1,0},
\end{align*}
where
\begin{align}\label{eq:ABC10}
\AC^{\pm}_{1,0}[T] = \begin{cases} (\D^{\T+}_{\T+}F|_{\sigma^-}-\nabla_-\sigma|_{F^{0-}\sigma^-}\D^{-}_{\T+}F|_{\sigma^-})\,T&\text{ if $\pm=-$}\\ T|_{F^{0+}\sigma}\,\D^0_0F|_{\sigma}&\text{ if $\pm=+$}\end{cases},
\end{align}
\begin{align*}
\BC^{\pm}_{1,0}[T] = \begin{cases} T|_{F^{0-}\sigma^-}\,\D^0_0F|_{\sigma^-}&\text{ if $\pm=-$}\\(\D^{\T-}_{\T-}F|_{\sigma}-\nabla_+\sigma^+|_{F^{0+}\sigma}\,\D^+_{\T-}F|_{\sigma})\,T&\text{ if $\pm=+$}\end{cases},
\end{align*}
and
\begin{align*}
\CC^{\pm}_{1,0} = \begin{cases}\nabla_-\sigma|_{F^{0-}\sigma^-}\,\D^-_0F|_{\sigma^-}-\D^{\T+}_0F|_{\sigma^-}&\text{ if $\pm=-$}\\\D^{\T-}_0F|_{\sigma} - \nabla_+\sigma^+|_{F^{0+}\sigma}\,\D^+_0F|_{\sigma}&\text{ if $\pm=+$}.\end{cases}
\end{align*}
Thus, for any nonzero $u\in E^{\pm}(r^{\pm})_{A^{\pm}}$,
\begin{align*}
\frac{|\nabla_0\sigma^{\pm}(u)|}{|u|}&\leq \frac{|(\AC^{\pm}_{1,0})^{-1}\BC^{\pm}_{1,0}[\nabla_0\sigma](u)|}{|u|} + \frac{|(\AC^{\pm}_{1,0})^{-1}\CC^{\pm}_{1,0}(u)|}{|u|}\\&\leq (\gamma^{\pm}_{1,0} + \delta^{\pm}_{1,0}(R,L_{1,\setminus0},L^{\pm}_{1,\setminus 0}))L_{1,0} + \epsilon^{\pm}_{1,0}(R,L_{1,\setminus0},L^{\pm}_{1,\setminus 0})
\end{align*}
with $\delta^{\pm}_{1,0}(R,L_{1,\setminus0},L^{\pm}_{1,\setminus 0}) = o_R(1)$ and $\epsilon^{\pm}_{1,0}(R,L_{1,\setminus0},L^{\pm}_{1,\setminus 0}) = O_R(1)$. Here we have used the facts that $|\sigma(u)|\leq L_{1,\setminus0}|u|$ and $|\sigma^{\pm}(u)|\leq L^{\pm}_{1,\setminus0}|u|$, from which it follows that $|(\AC^{\pm}_{1,0})^{-1}\CC^{\pm}_{1,0}(u)| = O(|u|)$ with constant depending on both $L^{\pm}_{1,\setminus0}$ and $L_{1,\setminus 0}$, since $\D F$ is locally Lipschitz; $|F^{0-}\sigma^-(u)|\leq (\nu_1+o_R(1))|u|$ with constant depending on $L_{1,\setminus 0}^-$ (when $\pm=-$); and $|(F^{0+}\sigma)^{-1}(u)|\leq (\nu_2^{-1}+o_R(1))|u|$ with constant depending on $L_{1,\setminus0}$ (when $\pm=+$).

If $k>1$, then for each $2\leq j\leq k$, inductive covariant differentiation of \eqref{eq:foliationidentities} yields a polynomial function
\[
\CC^{\pm}_j\equiv \CC^{\pm}_j(\D F,\dots,\D^j F;\nabla\sigma,\dots,\nabla^{j-1}\sigma;\nabla\sigma^{\pm},\dots,\nabla^{j-1}\sigma^{\pm}),
\]
whose restriction $\CC^{\pm}_{j,0}$ to $\T X^{\otimes j}$ vanishes on the zero section in $E^{\pm}$, such that
\begin{align*}
\AC^{\pm}_j[\nabla^j\sigma^{\pm}] = \BC^{\pm}_j[\nabla^j\sigma] + \CC^{\pm}_j,
\end{align*}
where
\begin{align*}
\AC^{\pm}_j[T] = \begin{cases} (\D^{\T+}_{\T+}F|_{\sigma^-} - \nabla\sigma|_{F^{0-}\sigma^-}\,\D^{0-}_{\T+}F|_{\sigma^-})\,T&\text{ if $\pm=-$}\\T|_{F^{0+}\sigma}\,(\D^{0+}_{0+}F|_{\sigma} + \D^{0+}_{\T-}F|_{\sigma}\,\nabla\sigma)^{\otimes j}&\text{ if $\pm=+$}\end{cases}
\end{align*}
and
\begin{align*}
\BC^{\pm}_j[T] = \begin{cases} T|_{F^{0-}\sigma^-}\,(\D^{0-}_{0-}F|_{\sigma^-}+\D^{0-}_{\T+}F|_{\sigma^-}\,\nabla\sigma^-)^{\otimes j}&\text{ if $\pm=-$}\\(\D^{\T-}_{\T-}F|_{\sigma}-\nabla\sigma^+|_{F^{0+}\sigma}\,\D^{0+}_{\T-}F|_{\sigma})\,T&\text{ if $\pm=+$}.\end{cases}
\end{align*}
It then similarly follows that
\begin{align*}
\|\nabla_{\setminus 0}^j\sigma^{\pm}\|_{\infty}&\leq \|(\AC^{\pm}_j)^{-1}\BC^{\pm}_{j}[\nabla^j\sigma]|_{\VC^j_{\pm}}\|_{\infty} + \|(\AC^{\pm}_j)^{-1}\CC^{\pm}_j|_{\VC^j_{\pm}}\|_{\infty}\\&\leq (\gamma^{\pm}_{j,\setminus0} + \delta^{\pm}_{j,\setminus0}(R,L_{1,0},L_{1,\setminus0}))L_{j,\setminus 0} + \epsilon^{\pm}_{j,\setminus0}(R,\Lb_{1:j-1,0},\Lb_{1:j-1,\setminus0},\Lb_{1:j-1,0}^{\pm},\Lb_{1:j-1,\setminus0}^{\pm})
\end{align*}
with $\delta^{\pm}_{j,\setminus0}(R,L_{1,0},L_{1,\setminus0}) = o_R(1)$ and $\epsilon^{\pm}_{j,\setminus0}(R,\Lb_{1:j-1,0},\Lb_{1:j-1,\setminus0},\Lb_{1:j-1,0}^{\pm},\Lb_{1:j-1,\setminus0}^{\pm}) = O_R(1)$. To estimate the horizontal derivatives, observe using the fact that $F$ is a bundle map that
\begin{align}\label{eq:ABCj0}
\AC^{\pm}_{j,0}[\nabla^j_0\sigma^{\pm}] = \BC^{\pm}_{j,0}[\nabla^j_0\sigma] + \hat{\CC}^{\pm}_{j,0}
\end{align}
where
\begin{align*}
\AC^{\pm}_{j,0}[T] = \begin{cases}(\D^{\T+}_{\T+}F|_{\sigma^-} - \nabla_-\sigma|_{F^{0-}\sigma^-}\,\D^-_{\T+}F|_{\sigma^-})T&\text{ if $\pm=-$}\\T|_{F^{0+}\sigma}\D^0_0F|_{\sigma}^{\otimes j}&\text{ if $\pm=+$}\end{cases},
\end{align*}
\begin{align*}
\BC^{\pm}_{j,0}[T] = \begin{cases} T|_{F^{0-}\sigma^-}\,\D^0_0F|_{\sigma^-}^{\otimes j}&\text{ if $\pm=-$}\\(\D^{\T-}_{\T-}F|_{\sigma}-\nabla_+\sigma^+|_{F^{0+}\sigma}\,\D^+_{\T-}F|_{\sigma})\,T&\text{ if $\pm=+$}\end{cases}
\end{align*}
and $\hat{\CC}^{\pm}_{j,0} = \CC^{\pm}_{j,0} +\DC^{\pm}_{j,0}$ where
\begin{align*}
\DC^{\pm}_{j,0} = \begin{cases} \BC^-_j[\nabla^j_{\setminus0}\sigma]|_{\T X^{\otimes j}} = \nabla^j_{\setminus 0}\sigma|_{F^{0-}\sigma^-}(\D^{0-}_0F|_{\sigma^-} + \D^{0-}_{\T+}F|_{\sigma^-}\,\nabla_0\sigma^-)^{\otimes j}&\text{ if $\pm=-$}\\-\AC^+_j[\nabla^j_{\setminus0}\sigma^+]|_{\T X^{\otimes j}} = -\nabla^j_{\setminus0}\sigma^+|_{F^{0+}\sigma}\,(\D^{0+}_0F|_{\sigma} + \D^{0+}_{\T-}F|_{\sigma}\,\nabla_0\sigma)^{\otimes j}&\text{ if $\pm=+$}.\end{cases}
\end{align*}
Consequently, for any nonzero $u\in E^{\pm}(r^{\pm})_{A^{\pm}}$, one has
\begin{align}
\frac{|\nabla^j_0\sigma^{\pm}(u)|}{|u|}&\leq\frac{|(\AC^{\pm}_{j,0})^{-1}\BC^{\pm}_{j,0}[\nabla^j_0\sigma](u)|}{|u|}+ \frac{|(\AC^{\pm}_{j,0})^{-1}\hat{\CC}^{\pm}_{j,0}(u)|}{|u|}\\&\leq(\gamma^{\pm}_{j,0} + \delta^{\pm}_{j,0}(R,L_{1,\setminus0},L^{\pm}_{1,\setminus0}))L_{j,0} + \epsilon^{\pm}_{j,0}(R,\Lb_{1:j-1,0},\Lb_{1:j,\setminus0},\Lb_{1:j-1,0}^{\pm},\Lb_{1:j,\setminus0}^{\pm})
\end{align}
with $\delta^{\pm}_{j,0}(R,L_{1,\setminus0},L^{\pm}_{1,\setminus0}) = o_R(1)$ and $\epsilon^{\pm}_{j,0}(R,\Lb_{1:j-1,0},\Lb_{1:j,\setminus0},\Lb_{1:j-1,0}^{\pm},\Lb_{1:j,\setminus0}^{\pm}) = O_R(1)$. To obtain the $\delta^{\pm}_{j,0}$ bounds, as in the $j=1$ case we have used $|F^{0-}\sigma^-(u)|\leq (\nu_1 + o_R(1))|u|$ with constant depending on $L^-_{1,\setminus0}$ (when $\pm=-$) and $|(F^{0+}\sigma)^{-1}(u)|\leq (\nu_2^{-1}+o_R(1))|u|$ with constant depending on $L_{1,\setminus0}$ (when $\pm=+$); to obtain the $\epsilon^{\pm}_{j,0}$ bounds, we have used the facts that $\D^jF$ is locally Lipschitz, $\CC^{\pm}_{j,0}$ vanishes on the zero section and $\DC^{\pm}_{j,0}(u)$ always contains at least one factor of the form $(\D^-_0F|_{\sigma^-} + \D^-_{\T+}F|_{\sigma^-}\,\nabla_0\sigma^-)(u) = O(|u|)$ with constant depending on $L_{j,\setminus0}$ (if $\pm=-$) or $(\D^+_0F|_{\sigma} + \D^+_{\T-}F|_{\sigma}\,\nabla_0\sigma)(u)= O(|u|)$ with constant depending on $L^+_{j,\setminus0}$ (if $\pm=+$).

\noindent\textbf{Lipschitz estimate: } The Lipschitz estimates for the map $\sigma\mapsto\sigma^{\pm}$ follow almost verbatim from the manifold argument. Let $\sigma_1,\sigma_2\in \SC^{\pm}(A,R,\Lb_0,\Lb_{\setminus0})$, with associated fixed sections $\tau^{\pm}_1,\tau^{\pm}_2$ of $\Psi^{\pm}_{\sigma_1},\Psi^{\pm}_{\sigma_2}$ and projections $\sigma^{\pm}_1,\sigma^{\pm}_2$ to sections of $\MC\rightarrow E^{\pm}$ respectively. Given $u^{\pm}\in E^{\pm}(r^{\pm})_{A^{\pm}}$, let $w_i^{\pm}$ be the projection to $\WC^{\pm}$ of $\tau^{\pm}_i(u^{\pm})$, so that $w^{\pm}_i= \Psi^{\pm}_{\sigma_i}(u^{\pm},w^{\pm}_i)$ for $i=1,2$. Then
\begin{align*}
|\sigma^{\pm}_1(u^{\pm})-\sigma^{\pm}_2(u^{\pm})|\leq |w^{\pm}_1-w^{\pm}_2|&\leq |\Psi^{\pm}_{\sigma_1}(u^{\pm},w^{\pm}_1)-\Psi^{\pm}_{\sigma_1}(u^{\pm},w^{\pm}_2)|+|\Psi^{\pm}_{\sigma_1}(u^{\pm},w^{\pm}_2)-\Psi^{\pm}_{\sigma_2}(u^{\pm},w^{\pm}_2)|\\&\leq q^{\pm}(R,L_{1,\setminus0})|w^{\pm}_1-w^{\pm}_2|+ |\Psi^{\pm}_{\sigma_1}(u^{\pm},w^{\pm}_2)-\Psi^{\pm}_{\sigma_2}(u^{\pm},w^{\pm}_2)|.
\end{align*}
Letting $v^{\pm}_{x^{\mp}}$ denote the projection of $w^{\pm}_2$ to $g_{\mp}^*E^{\pm}$, it follows from the definition of $\Psi^{\pm}_{\sigma_2}$ that
\begin{align*}
(1-q^{\pm}(R,L_{1,\setminus0}))|w^{\pm}_1-w^{\pm}_2|&\leq |(L^{\pm})^{-1}(\Phi^{\pm}_{\sigma_1}(u^{\pm},w^{\pm}_2)-\Phi^{\pm}_{\sigma_2}(u^{\pm},w^{\pm}_2))|\\&\leq \begin{cases} |(L^-)^{-1}(\sigma_1(v^-_{x^+})-\sigma_2(v^-_{x^+}))|&\text{ if $\pm=-$}\\|(L^+)^{-1}(F\circ\sigma_1(v^+_{x^-})-F\circ\sigma_2(v^+_{x^-}))|&\text{ if $\pm=+$}\end{cases}\\&\leq \begin{cases} (\min\{\mu_1,\nu_2\}^{-1}+o_R(1))\|\sigma_1-\sigma_2\|_*|v^-_{x^+}|&\text{ if $\pm=-$}\\ (\max\{\mu_2,\nu_1\}+o_R(1))\|\sigma_1-\sigma_2\|_*|v^+_{x^-}|&\text{ if $\pm=+$}\end{cases}.
\end{align*}
One moreover has
\[
|v^{\pm}_{\mp}|\leq \begin{cases} (\nu_1 + o_R(1))|u^-_{x^-}|&\text{ if $\pm=-$}\\(\nu_2^{-1}+o_R(1))|u^+_{x^+}|&\text{ if $\pm=+$}.\end{cases}
\]
One thus has
\[
\frac{|\sigma^{\pm}_1(u^{\pm})-\sigma^{\pm}_2(u^{\pm})|}{|u^{\pm}|}\leq \frac{|w^{\pm}_1-w^{\pm}_2|}{|u^{\pm}_{x^{\pm}}|}\leq \kappa^{\pm}(R,L_{1,\setminus0})\|\sigma_1-\sigma_2\|_*
\]
as claimed.
\end{proof}

We come now to the construction of an analytic graph transform analogous to that of the previous section. Given relatively open, relatively pre-compact $A^-\subset \mathrm{dom}(g)$, with $A^+:=g(A^-)$, choose a relatively open, relatively pre-compact $B^-\subset\mathrm{dom}(g)$ such that $\overline{A^-}\subset B^-$, set $B^+:=g(B^-)$ and $B:=B^-\cup B^+$. Given $R>0$ sufficiently small for Proposition \ref{prop:geometricgraphtransformsfoliation} to apply over $B$, let $r\equiv r(R)$ be as in the statement of the proposition. Choose a cutoff $b^{\pm}\in C^{\infty}_c(B^{\pm},[0,1])$ equal to 1 $A^{\pm}$ and equal to zero outside some neighbourhood of $\overline{A^{\pm}}$. Fix also a positive function $r_1 = o_r(r)$, $h>0$ and, using Lemma \ref{lem:logarithmiccutoff}, $0<r_0<r_1$ and a logarithmic cutoff $\varphi_h:[0,\infty)\rightarrow[0,1]$ satisfying \eqref{eq:logarithmiccutoffproperties}. Then define $\chi^{\pm}_h:E^{\pm}(R)_{B}\rightarrow[0,1]$ by extending
\[
\chi^{\pm}_h(w_x):=b^{\pm}(x)\,\varphi_h(|w_x|)
\]
by zero. Given $\Lb_{0},\Lb_{\setminus0}\in\RB^k_{\geq0}$, define the \emph{analytic graph transform} $\GC^{\pm}_h(\sigma)$ of any $\sigma\in\SC^{\pm}(B,R,\Lb_0,\Lb_{\setminus0})$ by
\[
\GC^{\pm}_h(\sigma)(w_x):=\begin{cases} (\chi^{\pm}_h\cdot\sigma^{\pm})(w_x)&\text{ if $w_x\in E^{\pm}(r)_{B^{\pm}}$}\\0&\text{ otherwise},\end{cases}
\]
where $\sigma^{\pm}$ is the geometric graph transform given by Proposition \ref{prop:geometricgraphtransformsfoliation}. Note that, with $\gamma^{\pm}_j$ defined as in \eqref{eq:gammaj}, one has $\gamma^{\pm}_j<1$ for all $j=0,\dots,k$ if and only if $\gamma^{\pm}_{j,0},\gamma^{\pm}_{j,\setminus0}<1$ for all $j=1,\dots,k$. One then has the following, again inspired by \cite{center_regularity}.

\begin{proposition}\label{prop:analyticgraphtransformfoliation}
Assume that $\gamma^{\pm}_{j}<1$ for all $j=0,\dots,k$. Then for all sufficiently small $R>0$, there exist $h>0$ and $\Lb_0,\Lb_{\setminus0}\in\RB^k_{\geq 0}$ such that $\GC^{\pm}_h$ maps $\SC^{\pm}(B,R,\Lb_0,\Lb_{\setminus0})$ into itself and is a contraction in the $\|\cdot\|_*$-norm. Consequently, $\GC^{\pm}_h$ admits a unique fixed point in $\SC^{\pm}(B,R,\Lb_0,\Lb_{\setminus0})$, and this fixed point is moreover $C^k$.
\end{proposition}

\begin{proof}
Fix $h>0$, $\Lb_0,\Lb_{\setminus0}\in\RB^k_{\geq0}$ and $R>0$ sufficiently small. Since $\chi^{\pm}_h$ takes values in $[0,1]$, any $\sigma_1,\sigma_2\in\SC^{\pm}(B,R,\Lb_0,\Lb_{\setminus0})$ satisfy
\[
\|\GC^{\pm}_h(\sigma_1)-\GC^{\pm}_h(\sigma_2)\|_*\leq \|\sigma^{\pm}_1-\sigma^{\pm}_2\|_*\leq \kappa^{\pm}(R,L_{1,\setminus0})\|\sigma_1-\sigma_2\|_*
\]
by Proposition \ref{prop:geometricgraphtransformsfoliation}. Taking $L_{1,\setminus0}$ and $R$ sufficiently small then gives $\kappa^{\pm}(R,L_{1,\setminus0})<1$ by Proposition \ref{prop:geometricgraphtransformsfoliation}.

To see that $h$, $\Lb_0,\Lb_{\setminus0}$ and $R$ can be chosen so that $\SC^{\pm}(B,R,\Lb_0,\Lb_{\setminus0})$ is invariant under $\GC^{\pm}_h$, fix $\sigma\in\SC^{\pm}(B,R,\Lb_0,\Lb_{\setminus0})$ and compute
\[
\nabla(\chi^{\pm}_h\sigma^{\pm}) = \chi^{\pm}_h\,\nabla\sigma^{\pm} + \nabla b^{\pm}\,\varphi_h\,\sigma^{\pm} + b^{\pm}\,\nabla\varphi_h\,\sigma^{\pm}.
\]
Since $\sigma^{\pm}$ vanishes to first order on the zero section, wherever defined, this equation implies that $\GC^{\pm}_h(\sigma)$ vanishes along the zero section also. To estimate the vertical derivatives, use the constancy of $b^{\pm}$ along the fibre, the observation that $\|\sigma^{\pm}\|_{*}\leq\|\nabla_{\setminus0}\sigma^{\pm}\|_{\infty}\leq L^{\pm}_{1,\setminus0}$ and \eqref{eq:logarithmiccutoffproperties} to estimate
\begin{align*}
\|\nabla_{\setminus0}\GC^{\pm}_h(\sigma)\|_{\infty}&\leq L^{\pm}_{1,\setminus0} + h\|\sigma^{\pm}\|_*\leq (1+h)L^{\pm}_{1\setminus0}\leq (1+h)(\gamma^{\pm}_{1,\setminus0}+\delta^{\pm}_{1,\setminus0}(R,L_{1,\setminus0}))L_{1,\setminus0}+\epsilon^{\pm}_{1,\setminus0}(R,L_{1,\setminus0}),
\end{align*}
with $\delta^{\pm}_{1,\setminus0},\epsilon^{\pm}_{1,\setminus0}= o_R(1)$ as in Proposition \ref{prop:geometricgraphtransformsfoliation}. Shrinking $R$ and $h$ if necessary, the hypothesis that $\gamma^{\pm}_{1,\setminus0}<1$ then gives $\|\nabla_{\setminus0}\GC^{\pm}_h(\sigma)\|_{\infty}\leq L_{1,\setminus0}$. For the horizontal derivative, note that the compatibility of the bundle metric and connection give $\nabla_0\varphi_h(|u|) = 0$, from which one obtains 
\begin{align*}
\|\nabla_0\GC^{\pm}_h(\sigma)\|_*\leq \Lb^{\pm}_{1,0} + \|\nabla b^{\pm}\|_{\infty}\|\sigma^{\pm}\|_*\leq (\gamma^{\pm}_{1,0} + \delta^{\pm}_{1,0}(R,L_{1,\setminus0}))L_{1,0} + \epsilon^{\pm}_{1,0}(R,L_{1,\setminus0}),
\end{align*}
with $\delta^{\pm}_{1,0} = o_R(1)$ and $\epsilon^{\pm}_{1,0} = O_R(1)$ as in Proposition \ref{prop:geometricgraphtransformsfoliation}. Taking $L_{1,0}$ larger and $R$ smaller if necessary, therefore, gives $\|\nabla_0\GC^{\pm}_h(\sigma)\|_{\infty}\leq L_{1,0}$.

Suppose now that $k\geq 2$, and recall the functions $\delta^{\pm}_{j,0}(R,L_{1,0},L_{1,\setminus0}),\delta^{\pm}_{j,\setminus0}(R,L_{1,0},L_{1,\setminus0}) = o_R(1)$ for $j=1,\dots,k$ of Proposition \ref{prop:geometricgraphtransformsfoliation}. Use the hypothesis that $\gamma^{\pm}_{j,0},\gamma^{\pm}_{j,\setminus0}<1$ for all $j=1,\dots,k$ to shrink $R$ sufficiently that
\begin{align}\label{eq:deltas}
\max_{j=1,\dots,k}\{\gamma^{\pm}_{j,0}+\delta^{\pm}_{j,0}(R,L_{1,0},L_{1,\setminus0}),\gamma^{\pm}_{j,\setminus0}+\delta^{\pm}_{j,\setminus0}(R,L_{1,0},L_{1,\setminus0})\}<1.
\end{align}
Assume as an inductive hypothesis that for some $2\leq j\leq k$ the map $\GC^{\pm}_h$ preserves $\SC^{\pm}(B,R,\Lb_{1:j-1,0},\Lb_{1:j-1,\setminus0})$. Observe that
\begin{align}\label{eq:derivativechisigma}
\nabla^j(\chi^{\pm}_h\sigma^{\pm}) = \chi^{\pm}_h\,\nabla^j\sigma^{\pm} + \sum_{i=1}^j\sum_{s\in\mathrm{Sh}(j,j-i)}s\big(\nabla^i\chi^{\pm}_h\otimes\nabla^{j-i}\sigma^{\pm}\big)
\end{align}
has only the first term containing a $j^{th}$ derivative of $\sigma^{\pm}$. By Proposition \ref{prop:geometricgraphtransformsfoliation} and the inductive hypothesis, there is $Q^{\pm}_{j,\setminus0}(R,\Lb_{1:j-1,0},\Lb_{1:j-1,\setminus0})<\infty$ such that
\begin{align*}
\|\nabla^j_{\setminus0}\GC^{\pm}_h(\sigma)\|_{\infty}\leq (\gamma^{\pm}_{j,\setminus0}+ \delta^{\pm}_{j,\setminus0}(R,L_{1,0},L_{1,\setminus0}))L_{j,\setminus0} + Q^{\pm}_{j,\setminus0}(R,\Lb_{1:j-1,0},\Lb_{1:j-1,\setminus0});
\end{align*}
invoking \eqref{eq:deltas} and increasing $L_{j,\setminus0}$ then gives $\|\nabla^j_{\setminus0}\GC^{\pm}_h(\sigma)\|_{\infty}\leq L_{j,\setminus0}$. Considering only horizontal derivatives, Proposition \ref{prop:geometricgraphtransformsfoliation} and the inductive hypothesis give $Q^{\pm}_{j,0}(R,\Lb_{1:j-1,0},\Lb_{1:j,\setminus0})$ such that
\begin{align*}
\|\nabla^j_{0}\GC^{\pm}_h(\sigma)\|_{\infty}\leq (\gamma^{\pm}_{j,0}+ \delta^{\pm}_{j,0}(R,L_{1,0},L_{1,\setminus0}))L_{j,0} + Q^{\pm}_{j,\setminus0}(R,\Lb_{1:j-1,0},\Lb_{1:j,\setminus0});
\end{align*}
again invoking \eqref{eq:deltas} and increasing $L_{j,0}$ if necessary then gives $\|\nabla^j_0\GC^{\pm}_h(\sigma)\|_{\infty}\leq L_{j,0}$ by the assumption that $\gamma^{\pm}_{j,0}<1$. Induction then gives invariance of $\SC^{\pm}(B,R,\Lb_0,\Lb_{\setminus0})$, so the Banach fixed point theorem gives a unique fixed section therein.

That this fixed section is $C^k$ now follows from a similar argument to that of Proposition \ref{prop:analyticgraphtransform}. Specifically, letting $\sigma^{\pm}_n$ be the $n^{th}$ iterate of the zero section under $\GC^{\pm}_h$, the sequence $\{\nabla^k\sigma^{\pm}_n\}_{n=1}^{\infty}$ is uniformly bounded by $\max\{L_{k,0}R,L_{k,\setminus0}\}$. An identical argument to that in Proposition \ref{prop:analyticgraphtransform} demonstrates that $\{\nabla^k_{\setminus0}\sigma^{\pm}_n\}_{n=1}^{\infty}$ is uniformly equicontinuous, so it remains only to demonstrate that $\{\nabla^k_0\sigma^{\pm}_n\}_{n=1}^{\infty}$ is uniformly equicontinuous. For this, define for each $n\in\NB$ the localised modulus
\[
\Omega_n(t,s):=\sup_{\substack{d^{\pm}(u,v)\leq t\\|u|,|v|\leq s}}|\nabla^k_0\sigma^{\pm}_n(u)-P_{v\rightarrow u}\nabla^k_0\sigma^{\pm}_n(v)|.
\]
It suffices to show that $\sup_{n\geq 0}\Omega_n(t,R) = o_t(1)$ as $t\rightarrow0$. Note that $\sup_{n\geq0}\Omega_n(t,s)\leq 2L_{k,0}s$. Denote
\[
c^{\pm}:=\begin{cases} \min\{\mu_1,\nu_2\}^{-1}\mu_2^k&\text{ if $\pm=-$}\\\max\{\mu_2,\nu_1\}\mu_1^{-k}&\text{ if $\pm=+$}.\end{cases}
\]
Recall the identities \eqref{eq:ABC10} and \eqref{eq:ABCj0} relating the derivatives of the geometric and analytic graph transforms when $k=1$ and $k>1$ respectively. Invoking \eqref{eq:derivativechisigma}, commuting the operator $(\AC^{\pm}_{k,0})^{-1}\BC^{\pm}_{k,0}$ of \eqref{eq:ABC10}, \eqref{eq:ABCj0} with parallel transport at the cost of a $o_t(1)$ error term, and noting the reparametrisation induced by this operator, yields $K_{\pm}>0$
\[
J_{\pm} =\begin{cases} \nu_1 + o_R(1)&\text{ if $\pm=-$}\\\nu_2^{-1}+ o_R(1)&\text{ if $\pm=+$}\end{cases}
\]
and, by the uniform equicontinuity of the mixed derivatives, $\beta_n(t,s)= o_t(1)$ uniformly in $n$ such that
\[
\Omega_{n+1}(t,s)\leq (c^{\pm} + o_R(1))\Omega_n(K_{\pm}t,J_{\pm}s) + \beta_n(t,s).
\]
Iterating this estimate $N$ times gives
\[
\Omega_n(t,s)\leq (c^{\pm}+o_R(1))^N\Omega_{n-N}(K_{\pm}^Nt,J_{\pm}^Ns) + \sum_{j=0}^{N-1}(c^{\pm}+o_R(1))^j\beta_n(K_{\pm}^jt,J_{\pm}^js).
\]
Shrinking $R$ if necessary, one has $J_{\pm}<1$ and $(c^{\pm}+o_R(1))J_{\pm} = \gamma^{\pm}_{k,0} +o_R(1)<1$, from which one obtains
\[
\Omega_n(t,s)\leq 2L_{k,0}(\rho^{\pm}_{k,0}+o_R(1))^Ns + \sum_{j=0}^{N-1}(c^{\pm}+o_R(1))^j\beta_n(K_{\pm}^jt,J_{\pm}^js).
\]
It now follows exactly as in the proof of Proposition \ref{prop:analyticgraphtransform} that for any $\epsilon>0$, one has $\sup_{n}\Omega_n(t,s)<\epsilon$ for all sufficiently small $t$, thus proving uniform equicontinuity of $\{\nabla^k_0\sigma^{\pm}_n\}_{n=1}^{\infty}$ and hence proving that the fixed point of $\GC^{\pm}_h$ is $C^k$.
\end{proof}

We can finally prove Theorem \ref{thm:invariantfoliations}.

\begin{theorem}\label{thm:invariantfoliationsmain}
If $\gamma^{\pm}_{j}<1$ for all $j=0,\dots,k$, then for any relatively open, relatively pre-compact $A^-\subset \dom(g)$, there is a locally invariant $C^k$ manifold $N^{\pm}$ containing $A^{\pm}$ a $C^k$, locally invariant foliation with leaves which are parametrised by $A^{\pm}$ and tangent at $A^{\pm}$ to $E^{\pm}_{A^{\pm}}$.
\end{theorem}

\begin{proof}
Proposition \ref{prop:analyticgraphtransformfoliation} gives $r>0$ and a $C^k$ section $\sigma^{\pm}_*:E^{\pm}(r)_{A^{\pm}}\rightarrow\MC_{A^{\pm}}$ which vanishes to first order on the zero section such that
\[
F\circ\sigma^{\pm}_* = \sigma^{\pm}_*\circ F^{\pm}\circ \sigma^{\pm}_*
\]
on a small tube in $E^{\pm}_{A^{\pm}\cap g^{-1}(A^{\pm})}$. Restricting the domain of $\sigma$ if necessary, the fact that for $0<r'<r$ one has $\nabla_0\sigma^{\pm}_*|_{E^{\pm}(r')_{A^{\pm}}} = O(r')$ then forces the submanifold $\Theta(\sigma^{\pm}_*(E^{\pm}(r)_{A^{\pm}})$ containing $A^{\pm}$ to be $C^k$ foliated by leaves $\mathcal{L}_x:=\Theta_x(\sigma^{\pm}_*(E^{\pm}(r)_x))$. This foliation has the desired properties: the leaves are tangent to $E^{\pm}$ since $\nabla\sigma^{\pm}_*(0_x) = 0$, and one has $f(\mathcal{L}_x)\cap \Theta(\MC(r)_{A^{\pm}})\subset \mathcal{L}_{f(x)}$ for all $x\in A^{\pm}\cap g^{-1}(A^{\pm})$.
\end{proof}

\section{Technical lemmas}\label{sec:technical}

The following logarithmic cutoff construction is a smoothed version of that presented on \cite[p. 54]{schoen_yau_positive_mass}.

\begin{lemma}\label{lem:logarithmiccutoff}
  For any $r_1,h>0$, there exists $0<r_0<r_1$ and a $C^{\infty}$ function $\varphi_{h}:[0,\infty)\rightarrow[0,1]$ such that $\varphi_h|_{[0,r_0]}\equiv 1$, $\varphi_h|_{[r_1,\infty)}\equiv 0$, and
  \[
  \sup_{s\in[0,\infty)}s|\nabla\varphi(s)|\leq h,
  \]
  where $\nabla\varphi$ denotes the ordinary Euclidean derivative.
\end{lemma}

\begin{proof}
Fix a $C^{\infty}$ cutoff function $\phi:\RB\rightarrow[0,1]$ such that $\phi|_{(-\infty,0]}\equiv 1$ and $\phi|_{[1,\infty)}\equiv 0$. Given $r_1,h>0$, define
\[
r_0:=r_1\exp(-\|\nabla\phi\|_{\infty}/h).
\]
Then define $\varphi_h:[0,\infty)\rightarrow[0,1]$ by the formula
\[
\varphi_h(s):=\phi\bigg(\frac{\log(s/r_0)}{\log(r_1/r_0)}\bigg),\qquad s>0,
\]
noting that since $\varphi_h|_{(0,r_0]}\equiv 1$, the formula admits a unique $C^{\infty}$ extension to $[0,\infty)$ satisfying $\varphi_h|_{[0,r_0]}\equiv 1$ and $\varphi_h|_{[r_1,\infty)}\equiv 0$. Moreover,
\[
s|\varphi_h(s)| = s\bigg|\nabla\phi\bigg(\frac{\log(s/r_0)}{\log(r_1/r_0)}\bigg)\frac{1}{s\log(r_1/r_0)}\bigg|\leq \frac{\|\nabla\phi\|_{\infty}}{\log(r_1/r_0)}\leq h
\]
as claimed.
\end{proof}

The following lemma is an application of the Arzela-Ascoli theorem to prove continuous differentiability of the uniform limit of a sequence of $C^1$ sections, which is used in several places throughout the document.

\begin{lemma}\label{lem:vbaa}
  Let $E\rightarrow M$ be a $C^{\infty}$ Euclidean vector bundle with compatible connection over a $C^{\infty}$ Riemannian manifold $M$. If $\{\sigma_n\}_{n=1}^{\infty}$ is a uniformly convergent sequence of $C^1$ sections of $E\rightarrow M$ with uniform limit $\sigma$. If $\{\nabla\sigma_n\}_{n=1}^{\infty}$ is uniformly bounded in $\|\cdot\|_{\infty}$ and uniformly equicontinuous, then $\sigma$ is $C^1$.
\end{lemma}

\begin{proof}
It suffices to work locally. Let $B$ be any geodesically strongly convex open subset of $M$ with compact closure $\overline{B}$. By the Arzela-Ascoli theorem, uniform boundedness and equicontinuity of $\{\nabla\sigma_n|_{\overline{B}}\}_{n=1}^{\infty}$ implies the existence of a subsequence of $\{\nabla\sigma_n|_{\overline{B}}\}_{n=1}^{\infty}$ with uniform limit $g$ some continuous section of $(T^*M\otimes E)|_{\overline{B}}\rightarrow\overline{B}$.

We prove that $g|_B$ is the derivative of $\sigma|_B$. Indeed, fixing $b_1,b_2\in B$, let $\gamma:[0,1]\rightarrow B$ the unique geodesic connecting $b_1$ and $b_2$, and, given $t\in[0,1]$ let $P_{\gamma(t)\rightarrow b_2}$ denote parallel transport from $\gamma(t)$ to $b_2$. The fundamental theorem of calculus applied to the vector-valued functions $[0,1]\ni t\mapsto P_{\gamma(t)}\nabla\sigma_n(\gamma(t))\in (T^*M\otimes E)_{b_2}$ gives
\[
\sigma_n(b_2)-P_{b_1\rightarrow b_2}\sigma_n(b_1) = \int_0^1P_{\gamma(t)\rightarrow b_2}\nabla\sigma_n(\gamma(t))[\dot{\gamma}(t)]\,dt.
\]
Uniform convergence of $\sigma_n$ to $\sigma$ and $\nabla\sigma_n$ to $g$ then imply that
\[
\sigma(b_2)-P_{b_1\rightarrow b_2}\sigma(b_1) = \int_0^1P_{\gamma(t)\rightarrow b_2}g(\gamma(t))[\dot{\gamma}(t)]\,dt.
\]
Since $b_1,b_2$ are arbitrary, this implies that $\sigma$ is differentiable, with $\nabla\sigma = g$; hence $\sigma$ is $C^1$.
\end{proof}

The next lemma is a version of the Banach fixed point theorem for vector bundles. It is used in the construction of the geometric graph transforms and follows straightforwardly from \cite[Theorem A.2]{deng}.

\begin{lemma}\label{lem:vbbanach}
Let $E\rightarrow M$ be a $C^{\infty}$ Euclidean vector bundle over a $C^{\infty}$ manifold $M$. Given $R>0$ and $k\in\NB\cup\{0\}$, suppose that $\Psi:\overline{E(R)}\rightarrow\overline{E(R)}$ is a $C^{k,1}$ morphism of fibre bundles. Given $e_x\in \overline{E(R)}$, with respect to the canonical identification of $\V_{e_x}E$ with $E_x$, write $\V_{e_x}\Psi$ for the almost-everywhere-defined vertical derivative of $\Psi$. Then, if there is $\alpha\in(0,1)$ such that $\|\V\Psi\|_{\infty}\leq\alpha$, there exists a unique section of $\overline{E(R)}\rightarrow M$ which is fixed by $\Psi$; moreover, this unique section is $C^{k,1}$.
\end{lemma}

\begin{proof}
It suffices to prove the result locally; uniqueness guarantees that the local fixed sections patch together into a global fixed section. Choose local coordinates for $M$ covered by a local, orthonormal trivialisation for $E$ over this coordinate patch. In these coordinates for $E$, $\V\Psi$ coincides with the Euclidean derivative with respect to the fibre; hence \cite[Theorem A.2]{deng} applies to give the desired local section.
\end{proof}

The next lemma allows us to deal with Lipschitz constants of sections by working instead with covariant derivatives.

\begin{lemma}\label{lem:lipschitzvsderivative}
Let $E\rightarrow M$ be a $C^{\infty}$ Euclidean vector bundle with metric connection over a $C^{\infty}$ Riemannian manifold $M$. Denote $M^{(2)}\subset M\times M$ for the neighbourhood of the diagonal $\Delta\subset M\times M$ consisting of points which may be connected by a unique length-minimising geodesic in $M$, and for any $(x,y)\in M^{(2)}$ let $P_{y\rightarrow x}:E_y\rightarrow E_x$ denote parallel transport along this geodesic. If $\sigma$ is a $C^{0,1}$ section of $E\rightarrow M$, denote
\[
\mathrm{Lip}(\sigma):=\sup_{(x,y)\in M^{(2)}\setminus\Delta}\frac{|\sigma(x)-P_{y\rightarrow x}\sigma(y)|}{d_M(x,y)}.
\]
Then $\Lip(\sigma) = \|\nabla\sigma\|_{\infty}$, where the right is interpreted as an essential supremum.
\end{lemma}

\begin{proof}
We first prove that $\Lip(\sigma)\leq\|\nabla\sigma\|_{\infty}$. Fix $(x,y)\in M^{(2)}\setminus\Delta$ and set $v:=\sigma(x)-P_{y\rightarrow x}\sigma(y)$ and assume without loss of generality that $v\neq 0$. Let $\mathrm{dom}(\varphi)$ be some open neighbourhood of $0$ in $\T_xM$ containing $\log_x(y)$ on which $\exp_x$ is a diffeomorphism onto its image, and define $\varphi:\mathrm{dom}(\varphi)\rightarrow E_x$ by
\[
\varphi(\xi_x):=P_{\exp_x(\xi_x)\rightarrow x}\sigma(\exp_x(\xi_x)),\qquad \xi_x\in\mathrm{dom}(\varphi).
\]
Observe then that
\[
|\sigma(x)-P_{y\rightarrow x}\sigma(y)| = \varphi(\log_x(y))-\varphi(0).
\]
By the Lebourg mean value theorem \cite[Theorem~2.3.7]{clarke}, there is $t\in(0,1)$ and a sequence $\{\xi_n\}_{n=1}^{\infty}\subset \T_x M$ converging to $t\log_x(y)$ of points on which $\varphi$ is differentiable such that
\begin{align}\label{eq:varphi}
\varphi(\log_x(y))-\varphi(0) = \lim_{n\rightarrow 0}\D\varphi(w_n)[\log_x(y)].
\end{align}
Letting $R_x$ denote the derivative of parallel transport $\xi_x\mapsto P_{\exp_x(\xi_x)\rightarrow x}$, one has
\[
\D\varphi(\xi_n)[\log_x(y)] = P_{\exp_x(\xi_n)\rightarrow x}\nabla\sigma(\exp_x(\xi_n))\,\T_{\xi_n}\exp_x[\log_x(y)] + R_x(\xi_n)[\log_x(y)]\sigma(\exp_x(\xi_n)),\qquad\forall n\in\NB.
\]
Since parallel transport is parallel along the radial direction and $\xi_n\rightarrow t\log_x(y)$, $R_x(\xi_n)[\log_x(y)]\rightarrow 0$; similarly $\T_{\xi_n}\exp_x(\log_x(y))\rightarrow \T_{t\log_x(y)}\exp_x(\log_x(y)) = (d/ds)|_{s=t}\exp_x(s\log_x(y))$ with norm $|\log_x(y)| = d_M(x,y)$. It then follows from \eqref{eq:varphi} that
\[
|\sigma(x)-P_{y\rightarrow x}\sigma(y)|\leq \|\nabla\sigma\|_{\infty}d_M(x,y).
\]
Taking the supremum over $(x,y)\in M^{(2)}\setminus\Delta$ gives the $\Lip(\sigma)\leq\|\nabla\sigma\|_{\infty}$.

For the converse inequality, Rademacher's theorem implies that
$\sigma$ is covariantly differentiable almost everywhere. Fix such a
point $x\in M$ and a unit vector $v\in\T_xM$. For all sufficiently small
$t>0$, set
\[
y_t:=\exp_x(tv).
\]
Then $(x,y_t)\in M^{(2)}$, $d_M(x,y_t)=t$, and the definition of the
covariant derivative gives
\[
\nabla\sigma(x)[v]
=
\lim_{t\rightarrow0}
\frac{
P_{y_t\rightarrow x}\sigma(y_t)-\sigma(x)
}{t}.
\]
It follows that
\[
|\nabla\sigma(x)[v]|
\leq
\Lip(\sigma).
\]
Taking the supremum over unit vectors $v\in\T_xM$, and then the
essential supremum over the points at which $\sigma$ is differentiable,
gives
\[
\|\nabla\sigma\|_{\infty}\leq\Lip(\sigma).
\]
Combining the two inequalities proves the result.
\end{proof}

Our final technical lemma provides an invariant complement to the tangent bundle of a locally invariant manifold, and is used to apply the stable foliation theorem to the relatively unstable manifold.

\begin{lemma}\label{lem:bundleextension}
Let $M$ be a $C^{\infty}$ Riemannian manifold, with $f:M\rightarrow M$ a $C^{k,1}$ map, $k\geq 1$. Let $X\subset Y\subset M$ be nested, embedded, diffeomorphically locally invariant manifolds and $E\rightarrow X$ a $C^{k-1,1}$ vector subbundle of $\T M_X$ which complements $\T Y_X$ and is locally invariant. Assume that one has spectral bounds
\[
\mu_1\leq s_{\min}(\T_xf|_{\T_xY})\leq s_{\max}(\T_xf|_{\T_xY})\leq \mu_2,\qquad s_{\max}(\T_xf|_{E_x})\leq \nu
\]
holding uniformly over $x\in X\cap f^{-1}(X)$. Then, if for all $j=0,\dots,k$ one has $\mu_1^{-1}\mu_2^j\nu<1$, for any relatively open, relatively pre-compact $A\subset X\cap f^{-1}(X)$ there is a neighbourhood $U$ of $A$ in $Y$ and a locally invariant $C^{k-1,1}$ bundle $E'\rightarrow U$ extending $E\rightarrow A$.
\end{lemma}

\begin{proof}
Given $A$, choose a relatively open, relatively pre-compact $B\subset X\cap f^{-1}(X)$ containing $\overline{A}$, a neighbourhood $V$ of $B\cup f(B)$ in $Y$ and a $C^{k-1,1}$ complement $\EC\rightarrow V$ of $\T Y_V$ extending $E\rightarrow X\cap V$. Relative to the splitting $\T M_{C}\cong \T Y_{V}\oplus \EC$, write
\[
\T_yf = \begin{pmatrix} \Af_y & \Bf_y \\ 0 & \Cf_y\end{pmatrix}
\]
for all $y\in V\cap f^{-1}(V)$. Since $E$ is invariant and $\EC_{X\cap V} = E_{X\cap V}$, one has
\[
\Bf_x = 0,\qquad \Cf_x = \T_xf|_{E_x},\qquad\forall x\in X\cap V\cap f^{-1}(V).
\]
Shrinking $V$ if necessary, the hypothesis that $\mu_1^{-1}\nu<1$ implies that $q:=\sup_{y\in V\cap f^{-1}(V)}|\Af_y^{-1}|\,|\Cf_y|<1$.

We seek an extension $E'$ of $E$ of the form $\mathrm{graph}(L)$ for some section $L$ of $\Hom(\EC,\T Y)$ over some neighbourhood of $A$ in $Y$. Local invariance of such an extension over a neighbourhood $U$ of $A$ in $V$ amounts to the identity
\begin{align}\label{eq:tangentinvariance}
L_y = \Af_y^{-1}(L_{f(y)}\,\Cf_y-\Bf_y),\qquad \forall y\in U\cap f^{-1}(U).
\end{align}
We demonstrate the existence of such a section via a graph transform argument. Choose a cutoff $\chi$ on $V$ with support in $V\cap f^{-1}(V)$, which is identically one on a neighbourhood $U$ of $A$ in $V$, and define an operator $\Gf$ on the Banach space of bounded continuous sections of $\Hom(\EC,\T Y)\rightarrow V$ by the formula
\[
\Gf(L)_y:=\chi(y)\,\Af_y^{-1}(L_{f(y)}\,\Cf_y-\Bf_y),\qquad y\in V,
\]
Since $\chi$ is supported in $V\cap f^{-1}(V)$, $X$ is locally invariant and $\EC|_{X\cap V}=E|_{X\cap V}$, this $\Gf$ preserves the closed subspace of sections which vanish along $X\cap V$. Moreover, $\|\Gf(L_1)-\Gf(L_2)\|_{\infty}\leq q\|L_1-L_2\|_{\infty}$, so $\Gf$ admits a unique continuous fixed section $L^*$ by the Banach fixed point theorem with $L*|_{X\cap V}\equiv 0$; thus the subbundle $E'$ associated to $L^*$ extends $E_A$. Since the cutoff is identically one on the neighbourhood $U$ of $A$ in $V$, \eqref{eq:tangentinvariance} holds over $U\cap f^{-1}(U)$ thus making $E'_U$ locally invariant. That $E'$ is moreover $C^{k-1,1}$ is proved by running an Arzela-Ascoli argument analogous to that of Proposition \ref{prop:analyticgraphtransform}: consider the $C^{k-1,1}$ iterates $L^n$ of the zero section under $\Gf$, and use the identities
\[
\nabla^j L^{n+1}_y = \chi(y)\,\Af_y^{-1}\,(\nabla^jL^n_{f(y)}\circ (\T_yf|_{\T_yY})^{\otimes j})\,\Cf_y+Q^j_n
\]
together with the estimates $\mu_1^{-1}\nu\mu_2^j<1$ to prove uniform boundedness (respectively, equicontinuity) of the $\{\nabla^jL^n\}_{n=1}^{\infty}$ for all $j=0,\dots,k$ (resp. all $j=0,\dots,k-1$) from which the desired regularity of $L^*$ follows.
\end{proof}

\section*{Declaration of AI assistance}\label{sec:AIuse}

This project was conducted using AI assistance for the following tasks:
\begin{enumerate}
  \item Locating relevant literature.
  \item Checking and discussing proposed arguments and finding counterexamples to fallacious arguments.
  \item Suggestion of alternative arguments when needed.
\end{enumerate}
Every word of the paper was typed by the author, and every argument, insofar as it was suggested by AI, underwent independent reconstruction by the author. The vision for and direction of the paper, including the theorem statements and the development of the unified framework presented in the paper for proving these statements, are due to the author with AI serving only in the roles indicated above. The author claims full responsibility for the content of the paper, including any errors therein.

\bibliographystyle{plain}
\bibliography{references}

\end{document}